\documentclass[11pt]{article}

\newif\ifarxiv
\arxivtrue

\ifarxiv
\usepackage{charter}
\usepackage{arxiv}
\fi

\usepackage[T1]{fontenc}
\usepackage[utf8]{inputenc}
\usepackage[english]{babel}
\usepackage{graphicx} 
\usepackage{wrapfig}
\usepackage{tabu}	

\usepackage{amsmath}
\usepackage{amsfonts}
\usepackage{amssymb}
\usepackage{amsthm}
\usepackage{mathrsfs}
\usepackage{stmaryrd}
\SetSymbolFont{stmry}{bold}{U}{stmry}{m}{n}
\usepackage[mathscr]{eucal}
\usepackage{dsfont}

\usepackage{algorithm}
\usepackage{algpseudocode}




\usepackage{graphicx}
\usepackage[dvipsnames]{xcolor}
\usepackage{hyperref}
\usepackage{subcaption}
\usepackage[nameinlink,noabbrev,capitalise]{cleveref}
\crefname{equation}{}{}
\Crefname{appendix}{Appendix}{Appendices}

\usepackage{enumerate}
\usepackage{enumitem}
\usepackage{fourier}
\usepackage{natbib}
\usepackage[many]{tcolorbox}
\tcbuselibrary{breakable, skins}

\usepackage{xspace}

\newcommand\R{\mathbb{R}}
\newcommand\Q{\mathbb{Q}}
\newcommand\E{\mathbb{E}}
\newcommand\N{\mathbb{N}}
\newcommand\Z{\mathbb{Z}}
\newcommand\Proba{\mathbb{P}}    

\newcommand\bfx{\mathbf{x}}
\newcommand\bfL{\mathbf{L}}

\newcommand\eqdef{\overset{\text{\tiny def\normalsize}}{=}}		
\newcommand{\normal}[2]{\ensuremath{\mathcal{N}\hspace*{-1pt}\left({#1},{#2}\right)}}			
\newcommand{\complexity}[1]{\ensuremath{\pmb{\mathcal{O}}\hspace*{-1pt}\left({#1}\right)}}		
\newcommand\eg{\emph{e.g.}\xspace}
\newcommand\ie{\emph{i.e.}\xspace}

\newcommand{\norm}[1]{\left\| #1 \right\|}

\DeclareMathOperator{\diam}{diam}  		
\DeclareMathOperator{\argmin}{argmin} 	
\DeclareMathOperator{\minimize}{minimize}

\newcommand*{\Quotient}[2]{\ensuremath{\raisebox{.2ex}{\ensuremath{#1}}\big/\!\raisebox{-.65ex}{\,\ensuremath{#2}}}}				

\newcommand\Xset{\mathscr{X}\xspace}		
\newcommand\bfpi{\boldsymbol\pi}			

\theoremstyle{plain}
\newtheorem{theorem}{Theorem}
\newtheorem{corollary}{Corollary}
\newtheorem{lemma}{Lemma}
\newtheorem{proposition}{Proposition}

\theoremstyle{definition}
\newtheorem{assumption}{Assumption}
\newtheorem{definition}{Definition}

\theoremstyle{remark}
\newtheorem{remark}{Remark}
\newtheorem{example}{Example}

\crefname{theorem}{theorem}{theorems}
\crefname{corollary}{corollary}{corollaries}
\crefname{lemma}{lemma}{lemmas}
\crefname{proposition}{proposition}{propositions}
\crefname{conjecture}{conjecture}{conjectures}
\crefname{claim}{claim}{claims}
\crefname{assumption}{assumption}{assumptions}
\crefname{definition}{definition}{definitions}
\crefname{remark}{remark}{remarks}
\crefname{example}{example}{examples}
\crefname{notation}{notation}{notations}

\newtcolorbox{Corner}{
  enhanced,
  breakable,
  colback=white,
  colframe=white,
  frame hidden, 
  grow to left by=-1em,
  boxsep=0pt,
  left=0pt, right=0pt,
  top=0pt, bottom=0pt,
  before skip=1.5em,
  after skip=1.5em
}

\newenvironment{Theorem}{\begin{Corner}\begin{theorem}}{\end{theorem}\end{Corner}}
\newenvironment{Corollary}{\begin{Corner}\begin{corollary}}{\end{corollary}\end{Corner}}
\newenvironment{Lemma}{\begin{Corner}\begin{lemma}}{\end{lemma}\end{Corner}}
\newenvironment{Proposition}{\begin{Corner}\begin{proposition}}{\end{proposition}\end{Corner}}

\newenvironment{Assumption}{\begin{Corner}\begin{assumption}}{\end{assumption}\end{Corner}}
\newenvironment{Remark}{\begin{Corner}\begin{remark}}{\end{remark}\end{Corner}}
\newenvironment{Example}{\begin{Corner}\begin{example}}{\hfill$\blacksquare$\end{example}\end{Corner}}

\definecolor{citationcolor}{RGB}{0,0,164}
\let\oldcite\cite
\renewcommand*\cite[1]{~\textcolor{citationcolor}{\oldcite{#1}}}
\let\oldcitet\citet
\renewcommand*\citet[1]{~\textcolor{citationcolor}{\oldcitet{#1}}}
\let\oldcitep\citep
\renewcommand*\citep[1]{~\textcolor{citationcolor}{\oldcitep{#1}}}

\newcommand\ST{Quadratic Minimum Spanning Tree\xspace}
\newcommand\TA{Traffic Assignment\xspace}

\title{First-Order Methods for\\Wasserstein Distributionally Robust Constrained Optimization}
\author{Hubert Villuendas\\
Université Grenoble Alpes, Inria, CNRS, LIG, LJK\\
\texttt{hubert.villuendas@univ-grenoble-alpes.fr}\\
\And
Mathieu Besan\c{c}on\\
Université Grenoble Alpes, Inria, CNRS, LIG\\
\texttt{mathieu.besancon@inria.fr}
\And
Jérôme Malick\\
Université Grenoble Alpes, CNRS, LJK\\
\texttt{jerome.malick@univ-grenoble-alpes.fr}
}
\date{}
\begin{document}

\maketitle
\title{First-Order Methods for Wasserstein Distributionally Robust Constrained Optimization}  

\begin{abstract}
We consider constrained problems in which input data are affected by errors. In such settings, Wasserstein distributionally robust optimization provides a principled framework to mitigate model risk by optimizing against worst-case data distributions within Wasserstein ambiguity sets.
However, the numerical resolution of the resulting problems remains challenging, especially in constrained settings.
In this paper, we provide a general, practical way to solve Wasserstein distributionally robust formulations in the presence of constraints. 
Our approach only requires a linear minimization oracle for the feasible set, and combines two key ingredients: (i) an entropic regularization of the distributionally robust value function, which makes it possible to compute stochastic gradient estimators, and (ii) a stochastic Frank–Wolfe algorithm, which minimizes the regularized robust objective while naturally handling constraints.
We illustrate the method, its tractability, and its interests against empirical risk minimization, on two problems: the traffic assignment and the minimum quadratic spanning tree.
\end{abstract}


\section{Introduction: decision, uncertainty, and examples}

\subsection{Context}

We consider a general class of data-driven constrained optimization problems of the form
\begin{equation}
\min_{\bfx\in\Xset}\quad\E_{\xi}\left[f(\bfx,\xi)\right]\label{eq: original stochastic problem}
\end{equation}
where $\bfx$ is the decision variable lying in the feasible convex set
$\Xset\subseteq\R^n$ and $\xi$ is an uncertain scenario lying in a set
of possible scenarios $\Xi\subseteq\R^d$.
We assume that the distribution of $\xi$ is unknown, and that we only have access to training data
$\widehat{\xi}_1,\dots,\widehat{\xi}_N \in \Xi$.
The standard data-driven approach consists in replacing the unknown distribution by
the empirical distribution
$\widehat{\Proba}_N\eqdef(\delta_{\widehat{\xi}_1}+\dots+\delta_{\widehat{\xi}_N})/N$.
This leads to Empirical Risk Minimization (ERM) which minimizes the average loss on the training dataset \citep{shapiro2021lectures,kleywegt2002sample}:
\begin{equation}
\min_{\bfx\in\Xset}\quad\dfrac{1}{N}\sum_{k=1}^N f(\bfx,\widehat{\xi}_k).\tag{ERM}\label{eq: ERM}
\end{equation}
While simple and widely used, ERM may produce over-optimistic decisions that overfit the observed scenarios and perform poorly on new data, especially when the training samples are limited, noisy, or unrepresentative (see \eg \citep{duchi2019variance,mohajerin2018data}). Distributionally robust optimization addresses this limitation by replacing the single empirical distribution with a worst-case distribution chosen in an ambiguity set around the data. A particularly appealing way to define such an ambiguity set is through optimal transport. Given a ground cost function $c:\Xi\times\Xi\rightarrow\R_+$, the associated Wasserstein distance between two
probability distributions $\Q_1$ and $\Q_2$ on $\Xi$ is defined by
\begin{equation}
W_c\left(\Q_1,\Q_2\right)\eqdef
\inf_{\substack{\pi\,\in\,\mathrm{Prob}\left(\Xi\times\Xi\right)\\
[\pi]_1=\Q_1,\; [\pi]_2=\Q_2}}
\int_{\Xi\times\Xi}c\left(\xi,\xi'\right)\mathrm{d}\pi\left(\xi,\xi'\right),
\label{eq: wasserstein distance}
\end{equation}
where $[\pi]_1$ and $[\pi]_2$ denote the two marginals of the transport plan $\pi$; see \eg \citep{COTFNT}. This leads to Wasserstein distributionally robust optimization (WDRO)
\begin{equation}
\min_{\bfx\,\in\,\Xset}
\sup_{\substack{\Q\,\in\,\mathrm{Prob}\left(\Xi\right)\\
W_c\left(\Q,\widehat{\Proba}_N\right)\leq\varrho}}
\E_{\zeta\sim\Q}\left[f(\bfx,\zeta)\right].
\tag{WDRO}\label{eq: WDRO primal form}
\end{equation}
Here, $\varrho\geq 0$ is the radius of the ambiguity set. This formulation combines expressive modeling flexibility (see \eg\citep{mohajerin2018data}) and statistical guarantees (see \eg\citep{le2025universal}).
The inner supremum in \Cref{eq: WDRO primal form} is an optimization problem over an infinite-dimensional space of probability distributions, and is generally
intractable in its primal form. Under mild assumptions, however, duality yields the reformulation \citep{mohajerin2018data,blanchet2019quantifying,gao2023distributionally}
\begin{equation}
\min_{\bfx\,\in\,\Xset}\inf_{\lambda\geq 0}
\lambda\varrho+
\E_{\widehat{\xi}\sim\widehat{\Proba}_N}\left[
\sup_{\zeta\,\in\,\Xi}
\left\lbrace f(\bfx,\zeta)-\lambda c\left(\widehat{\xi},\zeta\right)\right\rbrace
\right],
\label{eq: WDRO with duality}
\end{equation}
where $\lambda$ is the dual variable associated with the constraint
$W_c\left(\Q,\widehat{\Proba}_N\right)\leq\varrho$. This dual reformulation nicely removes the optimization over probability distributions, but the supremum over $\zeta$ typically induces the non-smoothness of the objective function. Thus, in general, solving \Cref{eq: WDRO with duality} is challenging, and existing solution approaches typically rely on additional structural assumptions on both $f$ and $\Xi$: indeed, in some specific situations, the supremum leads to convex optimization \citep{mohajerin2018data}  or sometimes admits a closed-form reformulation (see an example in \Cref{app: exact vs smoothed wdro}), but in general the inner sup cannot be easily tackled.

When the constraint set $\Xset$ is difficult to handle (\eg no compact algebraic representation or no easy-to-compute projection), the tractability of \Cref{eq: WDRO with duality} is even more complicated. Existing distributionally robust approaches for such constrained problems are therefore mostly problem-specific: see \eg  applications for vehicle routing\citep{ghosal2020distributionally},
facility location\citep{basciftci2021distributionally}, transit resource allocation\citep{sun2023distributionally}, or surgery scheduling \citep{chow2022target}. 
Among these, \citet{ghosal2020distributionally} uses a Benders decomposition, where the supremum is computed over a set of critical scenarios iteratively updated by a separation oracle that identifies worst-case situations. Another example, \citet{zhen2025unified} uses a standard Lagrangian reformulation applied to the inner maximization problem to yield a single minimization program. 

The goal of our work is to propose a \emph{general} method for Wasserstein distributionally robust constrained optimization, applicable to any problem with only a gradient oracle for the loss function and a linear minimization oracle for the feasible set. This requirement is satisfied for many important feasible sets (\eg flow polytopes, matching polytopes, matroid polytopes, the Birkhoff polytope, spectrahedra), opening many applications in operations research (network routing, matching, matrix completion, or submodular optimization); see the survey of \citet{braun2025conditionalgradientmethods}. In this paper, we illustrate our approach with the following two running examples.

\begin{Example}[\TA -- see \eg \citet{patriksson2015traffic}]
On a transportation network with origin-destination travel demands, the objective 
is to determine the flow on each arc while accounting for congestion. The mathematical formulation is described in \Cref{eq: formulation of the TA}, and its linear minimization oracle is described in \Cref{ex: LMO for TA}.
\end{Example}

\begin{Example}[\ST -- see \eg \citet{assad1992quadratic}]
For a given graph, this is 
a generalization of the classical Minimum Spanning Tree Problem where the objective function takes into account not only the cost of the selected edges, but also interaction costs induced by selecting pairs of edges. The problem is further formulated in \Cref{eq: formulation of the ST}, and its linear minimization oracle is described in \Cref{ex: LMO for ST}.
\end{Example}

\subsection{Contributions and outline}

This paper presents a practical approach to data-driven stochastic constrained optimization based on Wasserstein distributional robustness. 
As discussed above, two main computational difficulties arise in this setting: handling the worst-case scenario problem induced by the Wasserstein ambiguity set, and optimizing the resulting objective function over a hard-to-handle feasible region. 
We address these challenges simultaneously by combining two existing ideas (namely (i) entropic smoothing of the distributionally robust objective and (ii) conditional gradient methods for constrained optimization) into a unified and flexible framework. This 
algorithmic framework does not require tailoring to specific loss functions, uncertainty models, or algebraic representations of feasible sets.

Starting from the dual WDRO formulation in \Cref{eq: WDRO with duality}, we first replace the inner sample-wise supremum by an entropic log-sum-exp approximation, following \citet{azizian2023regularization,gao2023distributionally,vincent2024texttt}.
This yields the smoothed WDRO problem
\begin{equation}
\min_{\bfx\,\in\,\Xset}\inf_{\lambda\geq 0}\quad
F(\bfx,\lambda)\eqdef
\lambda\varrho
+
\varepsilon
\E_{\widehat{\xi}\sim\widehat{\Proba}_N}\left[
\log\left(
\E_{\zeta\sim \normal{\widehat{\xi}}{\sigma^2\mathbf{I}}}
\left[
\exp\left(
\dfrac{
f\left(\bfx,\zeta\right)-\lambda c\left(\widehat{\xi},\zeta\right)
}{\varepsilon}
\right)
\right]
\right)
\right],
\label{eq: regularized WDRO with duality}
\end{equation}
where $\varepsilon>0$ controls the smoothing level and $\sigma^2$ determines the sampling variance. This smoothing has both an algorithmic and a modeling role. Algorithmically, it turns the non-smooth dual objective into a differentiable surrogate amenable to stochastic first-order schemes. From a modeling viewpoint, it can be interpreted as an entropic regularization of WDRO, while preserving the statistical guarantees of the Wasserstein formulation (see \citep{azizian2023regularization,le2025universal}).

We then propose to solve \eqref{eq: regularized WDRO with duality}  with a momentum stochastic Frank-Wolfe scheme (see \eg \citep{mokhtari2020stochastic, braun2025conditionalgradientmethods}) which only requires access to a linear minimization oracle over the feasible set. This makes the method applicable to a broad class of constrained problems without deriving problem-specific robust reformulations.  We then provide arguments to make the smoothed WDRO formulation algorithmically analyzable. Although the dual formulation involves an unbounded multiplier $\lambda\in\R_+$, we prove that this apparent non-compactness is harmless, as the dual search can be explicitly restricted to a compact interval $[0,\bar\lambda]$. Then, we show that the stochastic gradients of the robust objective $F$ have uniformly bounded variance. These two estimates provide the technical arguments of the standard convergence analysis: compactness controls the size of the feasible set, while the variance bound controls the stochastic noise. 

We provide numerical illustrations on two problems where existing methods do not apply, namely \TA{} and \ST{}. We recover in this constrained setting the behavior of the unconstrained setting\citep{mohajerin2018data}: ERM tends to over-promise and under-perform, whereas Wasserstein distributionally robust modeling under-promises and over-performs on new test data.

This paper is organized as follows. \Cref{sec: Smooth distributionally robust modeling} introduces the smoothed WDRO model, its assumptions, and the associated gradient estimators. We establish the main regularity properties of this surrogate and derive Monte Carlo gradient estimators adapted to its structure. \Cref{sec: Stochastic Frank-Wolfe} presents the stochastic Frank-Wolfe algorithm and its convergence guarantees. \Cref{sec: Numerical Illustrations} illustrates the robustness obtained by our model on the two running problems. Finally, four Appendices provide complementary informations (generation of instances, special cases, auxiliary results on ERM, and two technical results).

\section{Smooth distributionally robust framework}\label{sec: Smooth distributionally robust modeling}

In this section, we present the smooth WDRO framework, adapted from \cite{azizian2023regularization, gao2023distributionally, vincent2024texttt} to treat the general constrained case.
\Cref{sec: Mathematical setup} presents our assumptions on the objective function, the constraints and scenario sets, and on the ground cost.  \Cref{sec: The robust objective and its properties} discusses the properties of the smoothed robust function and its gradient, key for the first-order method considered in \Cref{sec: Stochastic Frank-Wolfe}.

\subsection{Assumptions and examples}\label{sec: Mathematical setup}

This section presents the assumptions that will be made throughout the paper.

\begin{Assumption}[On the original objective function $f$]\label{assump: f is sufficiently smooth}
\begin{enumerate}[label=(\roman*),leftmargin=1em]
\item[]
\item For any scenario $\xi\in\Xi$, the function $\bfx\mapsto f(\bfx,\xi)$ is a $\mathcal{C}^2$ convex function.
\item For any decision $\bfx\in\Xset$, the function $\xi\mapsto\nabla_\bfx f(\bfx,\xi)$ is continuous on $\Xi$.
\end{enumerate}
\end{Assumption}

Since our approach will rely on gradient-based methods, we thus assume that the loss function $f$ is sufficiently smooth with respect to the decision variable. This assumption is relatively general and satisfied by a wide range of applications. In particular, it is verified by our two running examples.

\begin{Example}[\TA]
For a network $G=(V,A)$, with origin-destination travel demands, the objective of the \TA is to determine the flow on each arc while accounting for congestion. In our case, the travel time is dependent on uncertain exogenous factors (such as vehicle fleet composition or weather conditions), and we make our decision based on a finite amount of observations. For each $(i,j)\in V^2$, denote $q_{i\to j}\geq 0$ the traffic demand between $i$ and $j$, and $\mathcal{P}_{i\to j}$ the set of paths connecting the origin $i$ to the destination $j$. For a determinist scenario $\xi=((t^{(0)}_a)_{a\in A},\alpha,\beta)$, the \TA writes:
\begin{equation}
\underset{\bfx\in\Xset_{\mathrm{flow}}}{\minimize} \quad f(\bfx,\xi)=\displaystyle\sum_{a\in A}\int^{\bfx_a}_0t^{(0)}_a\left(1+\alpha\left(\dfrac{s}{c_a}\right)^\beta\right)\mathrm{d}s\label{eq: formulation of the TA}
\end{equation}
where $t^{(0)}_a>0$ denotes the free-flow travel time, $c_a>0$ is the capacity of arc $a$, and $\alpha,\beta>0$ are dimensionless tuning parameters, and $\Xset_{\mathrm{flow}}$ is the flow-polytope defined by: 
\begin{equation}
\Xset_{\mathrm{flow}}\eqdef\left\lbrace\bfx\in\R^{|A|}_+\;\left|\;\begin{array}{ll}
\displaystyle\sum_{p\in\mathcal{P}_{i\to j}}\sum_{a\in p}f_p\mathbb{1}_p(a)=q_{i\to j} & \forall (i,j)\in V^2,\\
\displaystyle\sum_{(i,j)\in V^2}\displaystyle\sum_{p\in\mathcal{P}_{i\to j}}f_p=\bfx_a & \forall a\in A,\\
f_p\geq 0 & \forall p\in\mathcal{P}_{i\to j},\forall (i,j)\in V^2
\end{array}\right.\right\rbrace.
\label{eq: TA polytope}
\end{equation}
The objective function is $\mathcal{C}^2$, and its gradient with respect to $\bfx_a$ is
$$\nabla_{\bfx_a}f(\bfx,\xi)=t^{(0)}_a\left(1+\alpha\left(\dfrac{\bfx_a}{c_a}\right)^{\beta}\right)$$
which is continuous with respect to $\xi=(\alpha,\beta)$. Thus $f$ verifies \Cref{assump: f is sufficiently smooth}.
\end{Example}

\begin{Example}\emph{\ST}. We consider the quadratic spanning tree problem over a graph $G=(V,E)$,
which includes bilinear cost terms for the selection of some pairs of edges. In our case, the costs are not known in advance, and we instead have access to a finite number of training samples. Here, the relaxed problem for a given cost scenario $\xi\in\Xi\subset\R_+^{m\times m}$ is written as
\begin{equation}
\underset{\bfx\in\Xset_{\mathrm{tree}}}{\minimize} \quad f(\bfx,\xi)=\bfx^\top\xi\bfx\label{eq: formulation of the ST}
\end{equation}
where $\Xset_{\mathrm{tree}}$ is the tree-polytope, defined by:
\begin{equation}
\Xset_{\mathrm{tree}}\eqdef\left\lbrace\bfx\in [0,1]^m
\;\left|\;\displaystyle\sum_{i=1}^m\bfx_i=n-1, \displaystyle\sum_{i\in E(S)}\bfx_i\leq |S|-1, \forall S\subseteq V \right.\right\rbrace,\label{eq: ST polytope}
\end{equation}
with $E(S)$ the set of edges adjacent to vertices in $S$.
The function $f(\cdot,\xi)$ is quadratic, thus $f(\cdot,\xi)\in\mathcal{C}^2$. Its derivative with respect to $\bfx$ is given by $\bfx\mapsto(\xi+\xi^\top)\bfx$, and thus for any fixed spanning tree embedding $\bfx\in\Xset$, the function $\xi\mapsto\nabla_\bfx f(\bfx,\xi)$ is a continuous map. Thus $f$ verifies the conditions of \Cref{assump: f is sufficiently smooth}.
\end{Example}

On the uncertain parameter space, we further assume that the set of all plausible scenarios $\Xi$ is closed and bounded. Moreover, we suppose that the decision space itself is compact, as is the case for our two running problems, as well as a wide variety of problems; see \eg \citet{bertsimas2011theory} for a more detailed survey.

\begin{Assumption}[On the constraints and scenario sets]\label{assump: Xi is compact}
\begin{enumerate}[label=(\roman*),leftmargin=1em]
\item[]
\item The set $\Xi\subset\R^d$ is compact and non-empty.
\item The set $\Xset$ is compact.
\end{enumerate}
\end{Assumption}

Note that the following quantity (which will appear in our analysis) is well-defined and finite, under \Cref{assump: f is sufficiently smooth,assump: Xi is compact}:
$$\Vert f\Vert_{\infty} \eqdef\sup_{\substack{\bfx\in\Xset, \xi\in\Xi}} |f(\bfx,\xi)|.$$

We also assume that the ground cost is comparable to the squared Euclidean norm. This controls the cost terms that appear in the dual variable and in the gradient of the regularized WDRO objective. This assumption is not used to apply the developments of this paper, but it is convenient to establish the mathematical properties and the convergence analysis. It is satisfied by many ground costs, including all the squared $p$-norm costs.

\begin{Assumption}[On the ground cost]\label{assump: continuity of the cost function for Wasserstein}
There exist constants $\mu_c,\bfL_c>0$ such that the ground cost $c:\Xi\times\Xi\rightarrow\R_+$ satisfies
$$
\mu_c\Vert\xi-\zeta\Vert_2^2\leq c(\xi,\zeta)\leq\bfL_c\Vert\xi-\zeta\Vert^2_2
\qquad \forall \xi,\zeta\in\Xi.
$$
In particular, $c(\xi,\xi)=0$ for all $\xi\in\Xi$.
\end{Assumption}

\subsection{Differentiability and gradient of robust objective}\label{sec: The robust objective and its properties}

The distributionally robust objective function \Cref{eq: regularized WDRO with duality} has useful properties that allow the employment of first-order methods, which we detail in this section. These results are already present, more or less explicitly, in the literature, in particular \cite{azizian2023regularization, vincent2024texttt}. We summarize the properties needed below and give a short proof for completeness.

\begin{Proposition}[Regularized WDRO objective]\label{prop: F is convex}\label{prop: F is differentiable}
Under \Cref{assump: f is sufficiently smooth,assump: Xi is compact,assump: continuity of the cost function for Wasserstein}, the function $F$ defined in \Cref{eq: regularized WDRO with duality} is convex and twice-differentiable on $\Xset\times\R_+$. Moreover,
\begin{align}
	\nabla_\bfx F(\bfx,\lambda)&=\E_{\widehat{\xi}\sim\widehat{\Proba}_N}\left[\E_{\zeta\sim\bfpi_{\bfx,\lambda}(\cdot|\widehat{\xi})}\left[\nabla_{\bfx}f(\bfx,\zeta)\right]\right],\label{eq: gradient in z}\\
	\dfrac{\partial}{\partial\lambda} F(\bfx,\lambda)&=\varrho -\E_{\widehat{\xi}\sim\widehat{\Proba}_N}\left[\E_{\zeta\sim\bfpi_{\bfx,\lambda}(\cdot|\widehat{\xi})}\left[c(\widehat{\xi},\zeta)\right]\right].\label{eq: gradient in lambda}
	\end{align}
where $\bfpi_{\bfx,\lambda}(\cdot|\xi)$ is the probability distribution given by
\begin{equation}\label{eq: definition of the distribution defining the robust gradient}
\mathrm{d}\bfpi_{\bfx,\lambda}(\cdot|\xi)\propto\exp\left(\Quotient{\left(f(\bfx,\zeta)-\lambda c(\xi,\zeta)\right)}{\varepsilon}\right)\exp\left(-\Quotient{\Vert\xi-\zeta\Vert^2_2}{2\sigma^2}\right)\mathrm{d}\zeta.
\end{equation}
\end{Proposition}

\begin{proof}[Proof of \Cref{prop: F is convex}]
For $k\in\llbracket N\rrbracket$, set
$$
h_k(\bfx,\lambda;\zeta)\eqdef\Quotient{f(\bfx,\zeta)-\lambda c(\widehat{\xi}_k,\zeta)}{\varepsilon},
\qquad
G_k(\bfx,\lambda)\eqdef\log\E_{\zeta\sim\normal{\widehat{\xi}_k}{\sigma^2\mathbf{I}}}\left[\exp\left(h_k(\bfx,\lambda;\zeta)\right)\right].
$$
For fixed $\zeta$, the map $(\bfx,\lambda)\mapsto h_k(\bfx,\lambda;\zeta)$ is convex. The log-sum-exp operator preserves convexity; equivalently, this follows from Hölder's inequality applied to $\exp(h_k)$ along any segment in $\Xset\times\R_+$. Hence each $G_k$ is convex, and so is $F$, as a nonnegative sum of the $G_k$ plus the linear term $\lambda\varrho$.
The compactness, smoothness, and cost-comparison assumptions allow differentiation under the integral sign for the exponential terms, and the normalizing denominator is strictly positive. Therefore,
\begin{align*}
	\nabla_\bfx F(\bfx,\lambda)&=\dfrac{1}{N}\sum_{k=1}^N\frac{\E_{\zeta\sim\normal{\widehat{\xi}_k}{\sigma^2\mathbf{I}}}\left[\nabla_{\bfx} f(\bfx,\zeta)\exp\left(h_k(\bfx,\lambda;\zeta)\right)\right]}{\E_{\zeta\sim\normal{\widehat{\xi}_k}{\sigma^2\mathbf{I}}}\left[\exp\left(h_k(\bfx,\lambda;\zeta)\right)\right]}=\E_{\widehat{\xi}\sim\widehat{\Proba}_N}\left[\E_{\zeta\sim\bfpi_{\bfx,\lambda}(\cdot|\widehat{\xi})}\left[\nabla_{\bfx}f(\bfx,\zeta)\right]\right],
\end{align*}
and similarly,
\begin{align*}
	\dfrac{\partial}{\partial\lambda} F(\bfx,\lambda)&=\varrho-\dfrac{1}{N}\sum_{k=1}^N\frac{\E_{\zeta\sim\normal{\widehat{\xi}_k}{\sigma^2\mathbf{I}}}\left[c(\widehat{\xi}_k,\zeta)\exp\left(h_k(\bfx,\lambda;\zeta)\right)\right]}{\E_{\zeta\sim\normal{\widehat{\xi}_k}{\sigma^2\mathbf{I}}}\left[\exp\left(h_k(\bfx,\lambda;\zeta)\right)\right]}=\varrho-\E_{\widehat{\xi}\sim\widehat{\Proba}_N}\left[\E_{\zeta\sim\bfpi_{\bfx,\lambda}(\cdot|\widehat{\xi})}\left[c(\widehat{\xi},\zeta)\right]\right].
\end{align*}
Applying the same differentiation argument to these gradient expressions
gives $F\in\mathcal{C}^2\left(\Xset\times\R_+\right)$.
\end{proof}

While equations \Cref{eq: gradient in lambda} provide a closed-form expression for the gradient of $F$, an exact computation remains numerically difficult. Indeed, these expressions involve multi-dimensional integrals on $\Xi$. To overcome this issue, we use a Monte Carlo sampling approach to approximate the integrals. In practice, for a sampling budget $S\in\N$, and for samples $\zeta^{(k)}_1,\dots,\zeta^{(k)}_S$ sampled from the Gaussian measure $\normal{\widehat{\xi}_k}{\sigma^2\mathbf{I}}$, for all $k\in\llbracket N\rrbracket$, one define the Monte-Carlo estimate of the gradient as:
\begin{equation}\label{eq: general formula for exact gradient}
\E_{\zeta\sim\bfpi_{\bfx,\lambda}(\cdot|\widehat{\xi}_k)}\left[\varphi_k(\bfx,\zeta)\right]\approx\sum_{s=1}^S\varphi_k(\bfx,\zeta^{(k)}_s)\dfrac{w^{(k)}_s}{\Vert w^{(k)}\Vert_1}
\end{equation}
where $\varphi_k(\bfx,\zeta^{(k)}_s)$ is to be replaced by $\nabla_{\bfx} f(\bfx,\zeta^{(k)}_s)$ (resp. $c(\widehat{\xi}_k,\zeta^{(k)}_s)$) to match the terms of \Cref{eq: gradient in z,eq: gradient in lambda}, and where $w^{(k)}_{s}\eqdef\exp\left(\Quotient{\left(f(\bfx,\zeta^{(k)}_s)-\lambda c(\widehat{\xi}_k,\zeta^{(k)}_s)\right)}{\varepsilon}\right)$ for $k\in\llbracket N\rrbracket$ and $s\in\llbracket S\rrbracket$.
This approximation can be intractable for problems where the dimension of the scenario is high, as is the case for the \ST\footnote{ Computing the gradient estimates in \Cref{eq: gradient estimate} for the \ST with a sampling budget $S\in\mathbb{N}$ requires drawing $N\times S$ samples of $m \times m$ real matrices. In a graph with $n$ vertices, the number of edges $m$ is typically $m=\complexity{n^2}$. Even for a relatively small graph where $n=50$ and $m=350$, evaluating the full-batch gradient with $S=10$ and a training set of size $N=100$ would involve generating and storing approximately 122 million scalars at each iteration, leading to a significant computational and memory burden.}. Thus, we adopt a mini-batch gradient scheme by restricting the gradient computation to a small subset of indices $\mathscr{J}\subseteq\llbracket N\rrbracket$, following\cite{vincent2024texttt}. Given a mini-batch $\mathscr{J}\subseteq\llbracket N\rrbracket$, and samples $\zeta^{(k)}_{1},\dots,\zeta^{(k)}_{S}\sim\normal{\widehat{\xi}_k}{\sigma^2\mathbf{I}}$ for all $k\in\mathscr{J}$, we define our estimator of \Cref{eq: gradient in z,eq: gradient in lambda} by taking the external sum over $\mathscr{J}$ only, and approximate each inner expectation as \Cref{eq: general formula for exact gradient}; this gives the following estimates:
\begin{equation}
{\mathcal{G}_{\bfx}}^{\mathscr{J}}(\bfx,\lambda)\eqdef\dfrac{1}{|\mathscr{J}|}\sum_{k\in\mathscr{J}}\sum^S_{s=1}\nabla_{\bfx} f(\bfx,\zeta^{(k)}_s)\dfrac{w^{(k)}_{s}}{\Vert w^{(k)}\Vert_1}\qquad
{\mathcal{G}_{\lambda}}^{\mathscr{J}}(\bfx,\lambda)\eqdef\varrho-\dfrac{1}{|\mathscr{J}|}\sum_{k\in\mathscr{J}}\sum^S_{s=1}c(\widehat{\xi}_k,\zeta^{(k)}_s)\dfrac{w^{(k)}_{s}}{\Vert w^{(k)}\Vert_1}.\label{eq: gradient estimate}
\end{equation}
where $w^{(k)}\in\R^S$ is the vector defined by $w^{(k)}_{s}\eqdef\exp\left(\Quotient{\left(f(\bfx,\zeta^{(k)}_s)-\lambda c(\widehat{\xi}_k,\zeta^{(k)}_s)\right)}{\varepsilon}\right)$ for $s\in\llbracket S\rrbracket$.

The resulting estimator is the practical object used by our first-order method. We record below the minimal consistency property needed for the stochastic Frank-Wolfe scheme introduced in \Cref{sec: Stochastic Frank-Wolfe}. 

\begin{Proposition}
Under \Cref{assump: f is sufficiently smooth,assump: Xi is compact}, the gradient estimator $\mathcal{G}^{\mathscr{J}}$ is asymptotically unbiased:
$$\E\left[\mathcal{G}^{\mathscr{J}}(\bfx,\lambda)\right]\xrightarrow[S\to\infty]{}\nabla F(\bfx,\lambda).$$
\end{Proposition}
\begin{proof}
We prove the result for the $\bfx$-component; the argument for the $\lambda$-component is identical. Here the expectation is taken over both the random mini-batch $\mathscr{J}$ and the Monte Carlo samples used to construct the estimator. Let $\mathcal{F}_S$ be the $\sigma$-algebra generated by all Monte Carlo samples $\left(\zeta_s^{(k)}\right)_{k\in\llbracket N\rrbracket,s\in\llbracket S\rrbracket}$ and define
$$
Y_{k,S}\eqdef\sum_{s=1}^S\nabla_{\bfx} f(\bfx,\zeta^{(k)}_s)\dfrac{w^{(k)}_s}{\Vert w^{(k)}\Vert_1},
\qquad k\in\llbracket N\rrbracket,
$$
which is deterministic conditionally on $\mathcal{F}_S$. Since the mini-batch $\mathscr{J}$ is sampled uniformly and independently of $\mathcal{F}_S$, the standard calculus gives (see \Cref{Lemma: mini-batch unbiasedness} in \Cref{app: technical lemmas} for a formal statements)
\begin{align*}
\E\left[\mathcal{G}_{\bfx}^{\mathscr{J}}(\bfx,\lambda)\right]
&=\E\left[\E\left[\frac{1}{|\mathscr{J}|}\sum_{k\in\mathscr{J}}Y_{k,S}\;\middle|\;\mathcal{F}_S\right]\right]
=\frac{1}{N}\sum_{k=1}^N\E\left[Y_{k,S}\right].
\end{align*}

It remains to identify the limit of each local Monte Carlo estimator $Y_{k,S}$. Fix $k$ and set
$$
\omega_k(\zeta)\eqdef\exp\left(\Quotient{f(\bfx,\zeta)-\lambda c(\widehat{\xi}_k,\zeta)}{\varepsilon}\right).
$$
The law of large numbers gives
$$
\frac{1}{S}\sum_{s=1}^S\nabla_{\bfx} f(\bfx,\zeta^{(k)}_s)\omega_k(\zeta^{(k)}_s)\xrightarrow[S\to\infty]{\mathrm{a.s.}}\E\left[\nabla_{\bfx} f(\bfx,\cdot)\omega_k(\cdot)\right],
\qquad
\frac{1}{S}\sum_{s=1}^S\omega_k(\zeta^{(k)}_s)\xrightarrow[S\to\infty]{\mathrm{a.s.}}\E\left[\omega_k(\cdot)\right].
$$
The denominator limit is finite by the definition of $F$ and strictly positive because $\omega_k>0$ almost surely. Hence the continuous mapping theorem yields
$$
Y_{k,S}\xrightarrow[S\to\infty]{\mathrm{a.s.}}
\frac{\E\left[\nabla_{\bfx} f(\bfx,\cdot)\omega_k(\cdot)\right]}{\E\left[\omega_k(\cdot)\right]}
=\E_{\zeta\sim\bfpi_{\bfx,\lambda}(\cdot|\widehat{\xi}_k)}\left[\nabla_{\bfx}f(\bfx,\zeta)\right].
$$

We now justify the passage from almost sure convergence to convergence of expectations. By \Cref{assump: f is sufficiently smooth,assump: Xi is compact}, the map $(\bfx,\zeta)\mapsto\nabla_{\bfx} f(\bfx,\zeta)$ is bounded on $\Xset\times\Xi$. Let
$$
C_z\eqdef\sup_{\substack{\bfx\in\Xset,\,\zeta\in\Xi}}\Vert\nabla_{\bfx}f(\bfx,\zeta)\Vert<+\infty.
$$
Since $w_s^{(k)}>0$ and $\displaystyle\sum_{s=1}^S \Quotient{w_s^{(k)}}{\Vert w^{(k)}\Vert_1}=1$, the estimator $Y_{k,S}$ is a convex combination of vectors whose norms are bounded by $C_z$. Hence, for every $S\geq1$,
$$
\Vert Y_{k,S}\Vert
\leq
\sum_{s=1}^S
\Vert\nabla_{\bfx}f(\bfx,\zeta_s^{(k)})\Vert
\frac{w_s^{(k)}}{\Vert w^{(k)}\Vert_1}
\leq C_z.
$$
Thus the family $(Y_{k,S})_{S\geq1}$ is uniformly bounded in $L^\infty$, and therefore uniformly integrable. Vitali's convergence theorem gives convergence of $\E[Y_{k,S}]$ for every $k$. Since $N$ is finite,
$$
\E\left[\mathcal{G}_{\bfx}^{\mathscr{J}}(\bfx,\lambda)\right]\xrightarrow[S\to\infty]{}
\frac{1}{N}\sum_{k=1}^N
\E_{\zeta\sim\bfpi_{\bfx,\lambda}(\cdot|\widehat{\xi}_k)}\left[\nabla_{\bfx}f(\bfx,\zeta)\right]
=\nabla_{\bfx}F(\bfx,\lambda).
$$
The same argument applies to the $\lambda$-component. Since $\Xi$ is compact and $c$ is bounded on $\Xi\times\Xi$ under \Cref{assump: continuity of the cost function for Wasserstein}, the constant
$$
C_c\eqdef\sup_{\xi,\zeta\in\Xi}c(\xi,\zeta)
$$
is finite. The self-normalized cost average is bounded by $C_c$, and therefore $\mathcal{G}_{\lambda}^{\mathscr{J}}(\bfx,\lambda)$ is bounded by $\varrho+C_c$, uniformly in $S$. This again gives uniform integrability and allows Vitali's theorem to identify the limit of the expectation. Consequently, both components converge to the corresponding components of $\nabla F(\bfx,\lambda)$.
\end{proof}

In the next section, we apply the stochastic Frank-Wolfe algorithm to tackle our robust problem.

\section{First-order methods tackling WDRO}\label{sec: Stochastic Frank-Wolfe}

Optimizing our distributionally robust objective \Cref{eq: regularized WDRO with duality} over constrained sets involves managing the geometry of the feasible set. To do so, we propose to use a Frank-Wolfe algorithm, following recent success in both ML and OR communities \citep{jaggi2013revisiting,besanccon2022frankwolfe,braun2025conditionalgradientmethods}. In this section, we first describe the proposed stochastic Frank-Wolfe algorithm, which uses momentum and mini-batch sampling to handle the gradient of the regularized objective function. We then analyze its convergence properties.

\subsection{Stochastic Frank-Wolfe at work}\label{subsec: stochastic frank-wolfe at work}

To address the constraints of the regularized WDRO problem \Cref{eq: regularized WDRO with duality}, we adopt a \emph{Frank-Wolfe} framework \citep{frank1956algorithm,jaggi2013revisiting,harchaoui2015conditional}. This method is adapted to smooth functions over a constrained feasible set $\Xset$ which admits a \emph{Linear Minimization Oracle} (LMO):
\begin{equation}
 \text{LMO}_{\Xset}(\mathbf{g}) \in \underset{\mathbf{s} \in \Xset}{\argmin} \langle\mathbf{g},\mathbf{s}\rangle.\label{eq: LMO}
\end{equation}
For a wide variety of constrained problems, this linear subproblem is tractable and can be solved efficiently via specialized methods \citep{besanccon2022frankwolfe,besanccon2025improved}. It is the case for our two running examples.

\begin{Example}[\TA]\label{ex: LMO for TA}
Optimizing over the flow polytope described in \Cref{eq: formulation of the TA}, the LMO corresponds to a linear network flow subproblem. By decomposing the total demand across pairs of origin-destination, the linear oracle reduces to finding the shortest paths on the network given the current gradient weights \citep{mitradjieva2013stiff}. This shortest path subproblem is efficiently solved using \emph{Dijkstra's} algorithm.
\end{Example}

\begin{Example}[\ST]\label{ex: LMO for ST}
Over the spanning-tree polytope, the linear oracle associated with the quadratic objective reduces to minimizing its linearization at the current point. This is a linear optimization problem over the spanning-tree polytope and can therefore be solved exactly by computing a minimum spanning tree, which can be achieved in polynomial time via \emph{Kruskal}'s algorithm. We underline that we have a highly tractable linear minimization oracle, even if the spanning-tree polytope is difficult to describe efficiently (indeed, \citet{pokutta2026symmetric} recently showed that every symmetric extended formulation for polytope \Cref{eq: ST polytope} has $\complexity{n^3}$ inequalities).
\end{Example}

The standard (stochastic) Frank-Wolfe algorithm is given in \Cref{alg: stochastic FW}: an iteration consists of a gradient estimation, a resolution of the LMO, and an update of the iterate in the obtained direction.

\begin{algorithm}[H]
\caption{\emph{(Informal) Stochastic Frank-Wolfe algorithm}}\label{alg: stochastic FW}
\begin{algorithmic}
\Require step sizes $0\leq\alpha_t\leq 1$
\For{$t=0$ to $\dots$}
\State $\widetilde{\nabla}F(y_t)\gets$ gradient estimator
\State $v_t\gets\text{LMO}_{\Xset}\left(\widehat{\nabla}F(\mathbf{y}_t)\right)$\Comment{$v_t$ is an optimal solution of the linear subproblem \Cref{eq: LMO}}
\State $\mathbf{y}_{t+1}\gets \mathbf{y}_t+\alpha_t\left(v_t-\mathbf{y}_t\right)$
\EndFor
\end{algorithmic}
\end{algorithm}

Since the iterates of \Cref{alg: Momentum stochastic FW} are formed as convex combinations of the optimal vertices returned by the linear oracle, the algorithm computes a solution in $\Xset$. Depending on the application, this fractional solution can either be used directly (as in the continuous \TA) or rounded via problem-specific heuristics\footnote{Constructing an explicit discrete feasible solution from such a relaxation is a widely studied problem in operations research, for which numerous primal heuristics, rounding schemes, and relaxation-based reconstruction procedures are available. Our framework is complementary to these techniques: we focus on solving the constrained WDRO relaxation, while an application-specific primal heuristic can subsequently be employed whenever an integer solution is required.}.

Stochastic Frank-Wolfe methods (\Cref{alg: stochastic FW}) suffer from the constant gradient noise, which can prevent convergence. While classical solutions typically rely on increasing the sampling budget $S$ at each iteration to control the gradient noise, such strategies lead to an increasing computational cost. Instead, we propose here to use the version of \cite{mokhtari2020stochastic}, integrating a momentum term. This indeed provides a stable estimation of the descent direction by taking into account the previous gradient direction, and ensures convergence without expanding the sample size. We thus propose a Stochastic Frank-Wolfe variant with momentum and mini-batching to handle our WDRO objective function:

\begin{algorithm}[H]
\caption{\emph{Momentum Stochastic Frank-Wolfe algorithm with mini-batch}}\label{alg: Momentum stochastic FW}
\begin{algorithmic}
\Require $\bfx_0\in\Xset$, $\lambda_{\max}$ an upper bound for the dual variable $\lambda$, $\lambda_0\in[0,\lambda_{\max})$, step sizes $0\leq\alpha_t\leq 1$ and momentum terms $\beta_t\in [0,1]$ for $t\geq 0$ with $\beta_0=1$. A budget $b\in\llbracket N\rrbracket$ for the mini-batch size, a sampling budget $S\in\N$ to compute integrals
\For{$t=0$ to $\dots$}
\State $\mathscr{J}_t\gets$ random subset of $\llbracket N\rrbracket$ with $|\mathscr{J}_t|=b$
\For{$k\in\mathscr{J}_t$}
\State sample $\zeta^{(k)}_{1},\dots,\zeta^{(k)}_{S}\sim\normal{\widehat{\xi}_k}{\sigma^2\mathbf{I}}$ i.i.d.
\State $w^{(k)}_{s}\gets\exp\left(\Quotient{\left(f(\bfx_t,\zeta^{(k)}_s)-\lambda_t c(\widehat{\xi}_k,\zeta^{(k)}_s)\right)}{\varepsilon}\right)$
\EndFor
\State $\mathcal{G}^{\mathscr{J}_t}(\bfx_t,\lambda_t)\gets$ Estimate given by \Cref{eq: gradient estimate}
\State $\widehat{\nabla}F(\bfx_t,\lambda_t)\gets\beta_t\mathcal{G}^{\mathscr{J}_t}_S(\bfx_t,\lambda_t)+(1-\beta_t)\widehat{\nabla}F(\bfx_{t-1},\lambda_{t-1})$
\State $v_t\gets{\argmin}_{v}$ LMO$\left(\widehat{\nabla}F(\bfx_t,\lambda_t)\right)$
\State $(\bfx_{t+1},\lambda_{t+1})\gets (\bfx_t,\lambda_t)+\alpha_t\left(v_t-(\bfx_t,\lambda_t)\right)$
\EndFor
\end{algorithmic}
\end{algorithm}

This algorithm has been applied in other contexts to build convex relaxations to combinatorial and mixed-integer optimization problems \citep{hendrych2025convex}.

\begin{Remark}\label{rem: resample during computation of the gradient}
A practical advantage of the sampling routine in \Cref{alg: Momentum stochastic FW} is its ability to handle cases where the uncertainty support $\Xi$ has known boundaries within $\mathbb{R}^d$. Indeed, a simple rejection sampling strategy can be integrated into the loop to avoid sampling infeasible instances.
For instance, in the context of the \TA, the components of $\zeta^{(k)}_s$ represent edge weights and must naturally satisfy non-negativity constraints (\ie, $\Xi \subseteq \mathbb{R}^{m\times m}_+$): if a generated sample $\zeta^{(k)}_s$ has negative coefficients, it can simply be discarded and resampled.
\end{Remark}

\subsection{Convergence analysis}

This section analyzes the convergence properties of \Cref{alg: Momentum stochastic FW} when applied to \eqref{eq: regularized WDRO with duality}, or more precisely, to its convex relaxation. Specifically, \Cref{thm: convergence of SFW} establishes the convergence of our method toward an optimal solution $(\bfx^\star, \lambda^\star)$ of the relaxed problem. Furthermore, the numerical experiments presented in \Cref{sec: Numerical Illustrations} demonstrate that the solution $\bfx^\star$ offers robustness properties when evaluated on unseen out-of-sample data.

The standard convergence theory for \emph{Frank-Wolfe} algorithms requires the objective function to be minimized over a convex and compact domain\citep{braun2025conditionalgradientmethods}. While the convex hull $\Xset$ is necessarily compact from \Cref{assump: Xi is compact} -- (ii), the dual variable $\lambda$ is defined over the interval $[0, +\infty)$. To ensure convergence of the algorithm, we must identify a sufficiently large upper bound $\lambda_{\max}$ such that the solution space can be restricted to the compact interval $[0, \lambda_{\max}]$ without losing the optimal solution.
This is the purpose of \Cref{prop: upper bound on lambda max}; the auxiliary log-sum-exp inequality used in its proof is stated in \Cref{lemma: log-sum-exp lowers the function} in \Cref{app: technical lemmas}.

\begin{Proposition}[Upper bound on $\lambda$]\label{prop: upper bound on lambda max} Let $m_c\eqdef\max_{k\in\llbracket N\rrbracket}\E_{\zeta\sim\normal{\widehat{\xi}_k}{\sigma^2\mathbf{I}}}\left[c(\widehat{\xi}_k,\zeta)\right]$
and suppose that $\varrho>m_c$.

Let $\lambda_{\max}\eqdef\Quotient{2\Vert f\Vert_{\infty}}{(\varrho-m_c)}$. Then for all $\bfx\in\Xset$:
$$\inf_{\lambda\in\R_+}\lambda\varrho+\E_{\widehat{\xi}\sim\widehat{\Proba}_N}\left[\varphi_{\varepsilon}(\bfx,\lambda,\widehat{\xi})\right]=\inf_{\lambda\in [0,\lambda_{\max}]}\lambda\varrho+\E_{\widehat{\xi}\sim\widehat{\Proba}_N}\left[\varphi_{\varepsilon}(\bfx,\lambda,\widehat{\xi})\right]$$
where $\varphi_{\varepsilon}$ is the function on $\Xset\times\R_+\times\Xi$ defined by
$$\varphi_{\varepsilon}(\bfx,\lambda,\xi)\eqdef\varepsilon\log\left(\E_{\zeta\sim\normal{\xi}{\sigma^2\mathbf{I}}}\left[\exp\left(\dfrac{f(\bfx,\zeta)-\lambda c(\xi,\zeta)}{\varepsilon}\right)\right]\right).$$
\end{Proposition}
\begin{proof}\label{proof of proposition: upper bound on lambda max}
Let $\bfx\in\Xset$ and $k\in\llbracket N\rrbracket$. An explicit calculation gives
$$\partial_{\lambda}\varphi_{\varepsilon}(\bfx,\lambda,\widehat{\xi}_k)=-\E_{\zeta\sim\bfpi_{\bfx,\lambda}(\cdot|\widehat{\xi}_k)}\left[c(\widehat{\xi}_k,\zeta)\right]$$
where $\bfpi_{\bfx,\lambda}(\cdot|\widehat{\xi}_k)$ is the distribution defined in \Cref{eq: definition of the distribution defining the robust gradient}. We bound $-\partial_{\lambda}\varphi_{\varepsilon}$ uniformly in $\bfx\in\Xset$ and $\left\lbrace\widehat{\xi}_k\right\rbrace_{k\in\llbracket N\rrbracket}$. We have:
\begin{align}
\E_{\zeta\sim\bfpi_{\bfx,\lambda}(\cdot|\widehat{\xi}_k)}\left[c(\widehat{\xi}_k,\zeta)\right]&=\E_{\zeta\sim\bfpi_{\bfx,\lambda}(\cdot|\widehat{\xi}_k)}\left[\dfrac{f(\bfx,\zeta)-f(\bfx,\zeta)+\lambda c(\widehat{\xi}_k,\zeta)}{\lambda}\right]\notag{}\\
&=\dfrac{1}{\lambda}\E_{\zeta\sim\bfpi_{\bfx,\lambda}(\cdot|\widehat{\xi}_k)}\left[f(\bfx,\zeta)\right]-\dfrac{\varepsilon}{\lambda}\E_{\zeta\sim\bfpi_{\bfx,\lambda}(\cdot|\widehat{\xi}_k)}\left[\dfrac{f(\bfx,\zeta)-\lambda c(\widehat{\xi}_k,\zeta)}{\varepsilon}\right]\label{eq: rewriting of the gradient in lambda}
\end{align}
The first term in \Cref{eq: rewriting of the gradient in lambda} is uniformly bounded by $\Quotient{\E_{\zeta\sim\bfpi_{\bfx,\lambda}(\cdot|\widehat{\xi}_k)}\left[\Vert f\Vert_{\infty}\right]}{\lambda}=\Quotient{\Vert f\Vert_{\infty}}{\lambda}$. To bound the second term, we apply \Cref{lemma: log-sum-exp lowers the function} in \Cref{app: technical lemmas} with function $g\eqdef\Quotient{(f(\bfx,\cdot)-\lambda c(\widehat{\xi}_k,\cdot))}{\varepsilon}$ and the distribution $\Q\eqdef\normal{\widehat{\xi}_k}{\sigma^2\mathbf{I}}\in\mathrm{Proba}(\Xi)$. In this setting, the lemma gives the explicit inequality
\begin{equation}
\E_{\zeta\sim\bfpi_{\bfx,\lambda}(\cdot|\widehat{\xi}_k)}\left[\dfrac{f(\bfx,\zeta)-\lambda c(\widehat{\xi}_k,\zeta)}{\varepsilon}\right]\geq\log\left(\E_{\zeta\sim\normal{\widehat{\xi}_k}{\sigma^2\mathbf{I}}}\left[\exp\left(\dfrac{f(\bfx,\zeta)-\lambda c(\widehat{\xi}_k,\zeta)}{\varepsilon}\right)\right]\right).\label{eq: instantiated log-sum-exp lower bound}
\end{equation}
Therefore:
\begin{align}
-\dfrac{\varepsilon}{\lambda}\E_{\zeta\sim\bfpi_{\bfx,\lambda}(\cdot|\widehat{\xi}_k)}\left[\dfrac{f(\bfx,\zeta)-\lambda c(\widehat{\xi}_k,\zeta)}{\varepsilon}\right]&\leq-\dfrac{\varepsilon}{\lambda}\log\left(\E_{\zeta\sim\normal{\widehat{\xi}_k}{\sigma^2\mathbf{I}}}\left[\exp\left(\dfrac{f(\bfx,\zeta)-\lambda c(\widehat{\xi}_k,\zeta)}{\varepsilon}\right)\right]\right)\notag{}\\
&\leq-\dfrac{1}{\lambda}\E_{\zeta\sim\normal{\widehat{\xi}_k}{\sigma^2\mathbf{I}}}\left[f(\bfx,\zeta)-\lambda c(\widehat{\xi}_k,\zeta)\right]\label{eq: bound by Jensen's}\\
&=-\dfrac{1}{\lambda}\E_{\zeta\sim\normal{\widehat{\xi}_k}{\sigma^2\mathbf{I}}}\left[f(\bfx,\zeta)\right]+\E_{\zeta\sim\normal{\widehat{\xi}_k }{\sigma^2\mathbf{I}}}\left[c(\widehat{\xi}_k,\zeta)\right]\notag{}\\
&\leq\dfrac{1}{\lambda}\Vert f\Vert_{\infty}+m_c\notag{}
\end{align}
where \eqref{eq: bound by Jensen's} uses Jensen's inequality on the convex function $-\log$.
Finally, we have
$$-\partial_{\lambda}\varphi_{\varepsilon}(\bfx,\lambda,\widehat{\xi}_k)\leq\dfrac{2\Vert f\Vert_{\infty}}{\lambda}+m_c.$$
Assuming $\varrho>m_c$ and setting $\lambda_{\max}\eqdef\Quotient{2\Vert f\Vert_{\infty}}{(\varrho-m_c)}$, we have $0\leq \varrho+\partial_{\lambda}\varphi_{\varepsilon}(\bfx,\lambda_{\max},\widehat{\xi}_k)$ for all $\bfx\in\Xset$ and $k\in\llbracket N\rrbracket$. In particular, for all $\bfx\in\Xset$:
\begin{equation}
0\leq \varrho+\E_{\widehat{\xi}\sim\widehat{\Proba}_N}\left[\partial_{\lambda}\varphi_{\varepsilon}(\bfx,\lambda_{\max},\widehat{\xi}_k)\right]=\partial_{\lambda}\left( \lambda\varrho+\E_{\widehat{\xi}\sim\widehat{\Proba}_N}\left[\varphi_{\varepsilon}(\bfx,\lambda_{\max},\widehat{\xi}_k)\right]\right)=\partial_\lambda F(\bfx,\lambda_{\max}).\label{eq: bounds on the partial derivative in lambda}
\end{equation}
Thanks to \Cref{prop: F is convex}, we have that $F(\bfx,\cdot)$ is convex for any fixed $\bfx\in\Xset$, thus \Cref{eq: bounds on the partial derivative in lambda} ensures that the dual minimizer $\lambda^\star$ on $\R_+$ may be found on $[0,\lambda_{\max}]$.
\end{proof}

\Cref{prop: upper bound on lambda max} proposes an upper bound on the optimal dual variable $\lambda^\star\leq\lambda_{\max}$. Under \Cref{assump: continuity of the cost function for Wasserstein}, this upper bound can be made explicit by replacing the quantity $m_c$ defined in the proposition.

\begin{Corollary}\label{coro: bound on lambda_max}
Under \Cref{assump: f is sufficiently smooth,assump: Xi is compact,assump: continuity of the cost function for Wasserstein}, if $\varrho>\bfL_c\sigma^2d$, then any optimal dual parameter $\lambda^\star$ of \Cref{eq: regularized WDRO with duality} can be found within $[0,\lambda_{\max}]$ with $\lambda_{\max}\eqdef\Quotient{2\Vert f\Vert_{\infty}}{(\varrho-\bfL_c\sigma^2d)}$.
\end{Corollary}
\begin{proof}
Let $k\in\llbracket N\rrbracket$. Then $\E_{\zeta\sim\normal{\widehat{\xi}_k}{\sigma^2\mathbf{I}}}\left[c(\widehat{\xi}_k,\zeta)\right]\leq\bfL_c\E_{\zeta\sim\normal{\widehat{\xi}_k}{\sigma^2\mathbf{I}}}\left[\Vert\widehat{\xi}_k-\zeta\Vert^2_2\right]$ by \Cref{assump: continuity of the cost function for Wasserstein}. We can change variable for $\zeta=\widehat{\xi}_k+\sigma\mathbf{Z}$ and $\Vert\widehat{\xi}_k-\zeta\Vert^2_2=\sigma^2\Vert\mathbf{Z}\Vert^2_2=\sigma^2(\mathbf{Z}^2_1+\dots+\mathbf{Z}^d_1)$, with $\mathbf{Z}_{i}\sim\normal{0}{1}$ for all $i\in\llbracket d\rrbracket$. In particular $\E_{\zeta\sim\normal{\widehat{\xi}_k}{\sigma^2\mathbf{I}}}\left[c(\widehat{\xi}_k,\zeta)\right]\leq\bfL_c\sigma^2\left(\E_{\normal{0}{1}}\left[\mathbf{Z}_1^2\right]+\dots+\E_{\normal{0}{1}}\left[\mathbf{Z}_d^2\right]\right).$ Integration by part gives $\E_{\normal{0}{1}}\left[\mathbf{Z}_i^2\right]=1$, such that $\E_{\zeta\sim\normal{\widehat{\xi}_k}{\sigma^2\mathbf{I}}}\left[c(\widehat{\xi}_k,\zeta)\right]\leq\bfL_c\sigma^2d$. Since this is true for all $k\in\llbracket N\rrbracket$, defining $m_c$ as in \Cref{prop: upper bound on lambda max}, we have $\Quotient{2\Vert f\Vert_{\infty}}{(\varrho-m_c)}\leq\lambda_{\max}\eqdef\Quotient{2\Vert f\Vert_{\infty}}{(\varrho-\bfL_c\sigma^2d)}$, and \Cref{prop: upper bound on lambda max} gives the wanted result.
\end{proof}

A first consequence of \Cref{prop: upper bound on lambda max,coro: bound on lambda_max} is the Lipschitz continuity of the gradients, which we formalize in the following lemma.

\begin{Lemma}\label{prop: F is L lipschitz}
Let \Cref{assump: f is sufficiently smooth,assump: Xi is compact} hold. Then there exists $\bfL_F>0$ such that 
$\nabla F$ is $\bfL_F$-Lipschitz on $\Xset\times [0,\lambda_{\max}]$.
\end{Lemma}
\begin{proof}
Using \Cref{assump: f is sufficiently smooth,assump: Xi is compact}, \Cref{prop: F is differentiable} shows that our robust objective $F$ is twice differentiable on the compact set $\Xset\times [0,\lambda_{\max}]$. Thus, there exists $\bfL_F>0$ such that
$$\sup_{(\bfx,\lambda)\in\Xset\times [0,\lambda_{\max}]}\Vert\nabla^2 F(\bfx,\lambda)\Vert\leq\bfL_F.$$
The \emph{mean value theorem} allows us to state that $\Vert\nabla F(\bfx_2,\lambda_2)-\nabla F(\bfx_1,\lambda_1)\Vert\leq\bfL_F\cdot\Vert (\bfx_2,\lambda_2)-(\bfx_1,\lambda_1)\Vert$ for any $(\bfx_1,\lambda_1),(\bfx_2,\lambda_2)\in\Xset\times [0,\lambda_{\max}]$, and $F$ is thus $\bfL_F$-smooth on $\Xset\times [0,\lambda_{\max}]$.
\end{proof}

The $\bfL_F$-smoothness of $F$ provides the theoretical foundation for the convergence of first-order methods \citep{nesterov2018lectures,braun2025conditionalgradientmethods}, and is in particular needed for the convergence guarantees of \cref{alg: stochastic FW}. Under this setting, we can establish the rate of convergence of the stochastic Frank-Wolfe algorithm.

\begin{Theorem}\label{thm: convergence of SFW}
Let \Cref{assump: Xi is compact,assump: f is sufficiently smooth,assump: continuity of the cost function for Wasserstein} hold, and suppose $\varrho>\bfL_c\sigma^2d$. With batch size $b$, sampling budget $S$, step-sizes $\alpha_t=2/(t+7)$ and momentum terms $\beta_t=4/(t+8)^{2/3}$, \Cref{alg: Momentum stochastic FW} gives iterates $(\bfx_t,\lambda_t)_{t\geq 0}\in\Xset\times [0,\lambda_{\max}]$ satisfying, for all $t\geq0$,
\[
\E\left[F(\bfx_t,\lambda_t)\right]-F(\bfx^\star,\lambda^\star)
\leq
\frac{Q_{\mathrm{corr}}}{\sqrt[3]{t+9}},
\]
where $(\bfx^\star,\lambda^\star)\in\Xset\times [0,\lambda_{\max}]$ is a minimizer of $F$ and the constant $Q_{\mathrm{corr}}$ is defined as
$$
Q_{\mathrm{corr}}
\eqdef
\max\left\lbrace
\sqrt[3]{9}\left(F(\bfx_0,\lambda_0)-F(\bfx^\star,\lambda^\star)\right),
\frac{\bfL_FD}{2}
+2D\max\left\lbrace
3\Vert\nabla F(\bfx_0,\lambda_0)\Vert,
\sqrt{16\mathsf{V}_{\max}+2\bfL_F^2D^2}
\right\rbrace
\right\rbrace.
$$

The constants appearing above are all well-defined: $\bfL_F$ is the Lipschitz constant for $\nabla F$, 
$D$ is the diameter of $\Xset\times [0,\lambda_{\max}]$, $C$ is such that $\Vert\nabla_{\bfx}f(\bfx,\cdot)\Vert\leq C$ on $\Xi$, $\mu_c$ is defined in the ground cost assumption, and $C_c\eqdef\sup_{\xi,\zeta\in\Xi}c(\xi,\zeta)<+\infty$. Finally
\[
\mathsf{V}_{\max}
\eqdef
\frac{4(C^2+C_c^2)}{bS}
\exp\left(\frac{4\Vert f\Vert_\infty}{\varepsilon}\right)
\left(1+\frac{4\lambda_{\max}\mu_c\sigma^2}{\varepsilon}\right)^{-d/2}
\left(1+\frac{2\lambda_{\max}\bfL_c\sigma^2}{\varepsilon}\right)^d
+
\frac{N-b}{b(N-1)}(C^2+C_c^2).
\]
\end{Theorem}

The proof relies on the fact that under the given hypothesis, we can give explicit bounds on the variance of the gradient estimator. This is developed in the next important technical lemma.

\begin{Lemma}[Corrected variance bound for the gradient estimator]\label{lemma: variance of the gradient estimator}
Under the notation and the assumptions of \Cref{thm: convergence of SFW}.
For a mini-batch $\mathscr{J}$ sampled uniformly without replacement among all subsets of $\llbracket N\rrbracket$ with size $b$, and for a Monte Carlo budget $S$, the $\bfx$-component of the gradient estimator satisfies the total variance bound
\[
\mathrm{Var}\left[\mathcal{G}_{\bfx}^{\mathscr{J}}(\bfx,\lambda)\right]
\lesssim
\frac{4C^2}{bS}
\exp\left(\frac{4\Vert f\Vert_\infty}{\varepsilon}\right)
\left(1+\frac{4\lambda\mu_c\sigma^2}{\varepsilon}\right)^{-d/2}
\left(1+\frac{2\lambda\bfL_c\sigma^2}{\varepsilon}\right)^d
+
\frac{N-b}{b(N-1)}C^2.
\]
The same bound holds for the $\lambda$-component after replacing $C$ by $C_c$. Consequently, for the full estimator,
\[
\mathrm{Var}\left[\mathcal{G}^{\mathscr{J}}(\bfx,\lambda)\right]
\lesssim
\frac{4(C^2+C_c^2)}{bS}
\exp\left(\frac{4\Vert f\Vert_\infty}{\varepsilon}\right)
\left(1+\frac{4\lambda\mu_c\sigma^2}{\varepsilon}\right)^{-d/2}
\left(1+\frac{2\lambda\bfL_c\sigma^2}{\varepsilon}\right)^d
+
\frac{N-b}{b(N-1)}(C^2+C_c^2)
\]
where $\lesssim$ means that the left-hand sides are bounded up to a multiplicative constant, independent of the parameters.
\end{Lemma}

\begin{proof}
We give the proof for the $\bfx$-component. The proof for the $\lambda$-component is identical, replacing the bounded vector $\nabla_{\bfx}f(\bfx,\cdot)$ by the bounded scalar $c(\widehat{\xi}_k,\cdot)$. For $k\in\llbracket N\rrbracket$, define the local self-normalized estimator
\[
\mathbf{G}_{k,S}(\bfx,\lambda)
\eqdef
\sum_{s=1}^S\nabla_{\bfx}f(\bfx,\zeta_s^{(k)})
\frac{w_s^{(k)}}{\Vert w^{(k)}\Vert_1},
\qquad
w_s^{(k)}
\eqdef
\exp\left(\frac{f(\bfx,\zeta_s^{(k)})-\lambda c(\widehat{\xi}_k,\zeta_s^{(k)})}{\varepsilon}\right),
\]
with $\zeta_s^{(k)}\sim\normal{\widehat{\xi}_k}{\sigma^2\mathbf{I}}$ independent across $s$ and $k$. The global estimator is
\[
\mathcal{G}_{\bfx}^{\mathscr{J}}(\bfx,\lambda)
=\frac{1}{b}\sum_{k\in\mathscr{J}}\mathbf{G}_{k,S}(\bfx,\lambda).
\]
Conditionally on the mini-batch $\mathscr{J}$, the variables $\mathbf{G}_{k,S}$ are independent across $k$. Therefore,
\[
\mathrm{Var}_{\mathrm{MC}}\left[
\mathcal{G}_{\bfx}^{\mathscr{J}}(\bfx,\lambda)
\middle|\mathscr{J}
\right]
=
\frac{1}{b^2}\sum_{k\in\mathscr{J}}
\mathrm{Var}_{\mathrm{MC}}\left[\mathbf{G}_{k,S}(\bfx,\lambda)\right]
\leq
\frac{1}{b}\sup_{k\in\llbracket N\rrbracket}
\mathrm{Var}_{\mathrm{MC}}\left[\mathbf{G}_{k,S}(\bfx,\lambda)\right].
\]
It remains to bound the local variance uniformly in $k$.

To do so, we fix $k$ and write
\[
w_k(\zeta)
\eqdef
\exp\left(\frac{f(\bfx,\zeta)-\lambda c(\widehat{\xi}_k,\zeta)}{\varepsilon}\right),
\qquad
\mu_k
\eqdef
\frac{\E\left[\nabla_{\bfx}f(\bfx,\zeta)w_k(\zeta)\right]}
{\E\left[w_k(\zeta)\right]},
\]
where the expectations are taken with respect to $\zeta\sim\normal{\widehat{\xi}_k}{\sigma^2\mathbf{I}}$. The classical delta-method linearization of the self-normalized estimator (see \eg \cite{owen2013montecarlo}) gives
\[
\mathrm{Var}_{\mathrm{MC}}\left[\mathbf{G}_{k,S}(\bfx,\lambda)\right]
\lesssim
\frac{1}{S}
\frac{\E\left[\Vert\nabla_{\bfx}f(\bfx,\zeta)-\mu_k\Vert^2w_k(\zeta)^2\right]}
{\E\left[w_k(\zeta)\right]^2}.
\]
Since $\Vert\nabla_{\bfx}f(\bfx,\cdot)\Vert\leq C$, we also have $\Vert\mu_k\Vert\leq C$, and hence
$
\Vert\nabla_{\bfx}f(\bfx,\zeta)-\mu_k\Vert^2\leq 4C^2.
$
Thus,
\[
\mathrm{Var}_{\mathrm{MC}}\left[\mathbf{G}_{k,S}(\bfx,\lambda)\right]
\lesssim
\frac{4C^2}{S}
\frac{\E\left[w_k(\zeta)^2\right]}
{\E\left[w_k(\zeta)\right]^2}.
\]
The numerator is bounded using the lower comparison
$c(\widehat{\xi}_k,\zeta)\geq\mu_c\Vert\widehat{\xi}_k-\zeta\Vert_2^2$
and $\vert f\vert\leq\Vert f\Vert_\infty$:
\[
\E\left[w_k(\zeta)^2\right]
\leq
\exp\left(\frac{2\Vert f\Vert_\infty}{\varepsilon}\right)
\E\left[
\exp\left(-\frac{2\lambda\mu_c}{\varepsilon}
\Vert\widehat{\xi}_k-\zeta\Vert_2^2\right)
\right].
\]
With $\zeta=\widehat{\xi}_k+\sigma Z$, $Z\sim\normal{0}{\mathbf{I}}$, this gives the uniform bound
\[
\E\left[w_k(\zeta)^2\right]
\leq
\exp\left(\frac{2\Vert f\Vert_\infty}{\varepsilon}\right)
\left(1+\frac{4\lambda\mu_c\sigma^2}{\varepsilon}\right)^{-d/2}.
\]
For the denominator, \Cref{assump: continuity of the cost function for Wasserstein} also gives the upper comparison
\[
c(\widehat{\xi}_k,\zeta)\leq \bfL_c\Vert\widehat{\xi}_k-\zeta\Vert_2^2.
\]
Since the cost appears with a negative sign in the exponential, we get
\[
w_k(\zeta)
\geq
\exp\left(-\frac{\Vert f\Vert_\infty}{\varepsilon}\right)
\exp\left(-\frac{\lambda\bfL_c}{\varepsilon}
\Vert\widehat{\xi}_k-\zeta\Vert_2^2\right)
\quad\text{and}\quad\E\left[w_k(\zeta)\right]^2
\geq
\exp\left(-\frac{2\Vert f\Vert_\infty}{\varepsilon}\right)
\left(1+\frac{2\lambda\bfL_c\sigma^2}{\varepsilon}\right)^{-d}.
\]
Combining the bounds above gives, uniformly in $k$, a bound on $\mathrm{Var}_{\mathrm{MC}}\left[\mathbf{G}_{k,S}(\bfx,\lambda)\right]$ and directly
\[
\E_{\mathscr{J}}\left[
\mathrm{Var}_{\mathrm{MC}}\left[
\mathcal{G}_{\bfx}^{\mathscr{J}}(\bfx,\lambda)
\middle|\mathscr{J}
\right]\right]
\lesssim
\frac{4C^2}{bS}
\exp\left(\frac{4\Vert f\Vert_\infty}{\varepsilon}\right)
\left(1+\frac{4\lambda\mu_c\sigma^2}{\varepsilon}\right)^{-d/2}
\left(1+\frac{2\lambda\bfL_c\sigma^2}{\varepsilon}\right)^d.
\]

It remains to account for the randomness of the mini-batch. Let
\[
m_{k,S}(\bfx,\lambda)\eqdef
\E_{\mathrm{MC}}\left[\mathbf{G}_{k,S}(\bfx,\lambda)\right],
\qquad
\bar m_S(\bfx,\lambda)
\eqdef
\frac{1}{N}\sum_{k=1}^N m_{k,S}(\bfx,\lambda).
\]
The finite-population variance formula for uniform sampling without replacement gives
\[
\mathrm{Var}_{\mathscr{J}}\left[
\frac{1}{b}\sum_{k\in\mathscr{J}}m_{k,S}(\bfx,\lambda)
\right]
=
\frac{N-b}{b(N-1)}
\cdot
\frac{1}{N}\sum_{k=1}^N
\left\Vert m_{k,S}(\bfx,\lambda)-\bar m_S(\bfx,\lambda)\right\Vert^2.
\]
Each $\mathbf{G}_{k,S}$ is a convex combination of gradients with norm at most $C$: Jensen's inequality gives $\Vert m_{k,S}(\bfx,\lambda)\Vert\leq C$. Therefore,
\[
\frac{1}{N}\sum_{k=1}^N
\left\Vert m_{k,S}(\bfx,\lambda)-\bar m_S(\bfx,\lambda)\right\Vert^2
\leq
\frac{1}{N}\sum_{k=1}^N
\left\Vert m_{k,S}(\bfx,\lambda)\right\Vert^2
\leq C^2.
\]
The law of total variance now gives
\[
\mathrm{Var}\left[\mathcal{G}_{\bfx}^{\mathscr{J}}(\bfx,\lambda)\right]
=
\E_{\mathscr{J}}\left[
\mathrm{Var}_{\mathrm{MC}}\left[
\mathcal{G}_{\bfx}^{\mathscr{J}}(\bfx,\lambda)
\middle|\mathscr{J}
\right]\right]
+
\mathrm{Var}_{\mathscr{J}}\left[
\E_{\mathrm{MC}}\left[
\mathcal{G}_{\bfx}^{\mathscr{J}}(\bfx,\lambda)
\middle|\mathscr{J}
\right]\right],
\]
which proves the announced bound for the $\bfx$-component.

For the $\lambda$-component, the estimator is a constant $\varrho$ minus a mini-batch average of self-normalized averages of $c(\widehat{\xi}_k,\zeta)$. Since $0\leq c(\xi,\zeta)\leq C_c$, the same proof applies with $C$ replaced by $C_c$. Summing the componentwise bounds gives the stated bound for the full gradient estimator.
\end{proof}

From the previous lemma, the proof of the convergence theorem then follows from existing convergence results \citep{braun2025conditionalgradientmethods}.

\begin{proof}[Proof of \Cref{thm: convergence of SFW}]
Use Theorem 4.12 in \citep{braun2025conditionalgradientmethods}. The assumptions needed there are satisfied because $F$ is convex and $\bfL_F$-smooth, $\Xset\times [0,\lambda_{\max}]$ has diameter $D$, and the stochastic gradients have uniformly bounded variance by \Cref{lemma: variance of the gradient estimator}. As \Cref{coro: bound on lambda_max} states, we can bound each $\lambda_t$ by $\lambda_{\max}\leq \Quotient{2\Vert f\Vert_{\infty}}{(\varrho-\bfL_c\sigma^2d)}$. Setting this as an upper bound for the iterates gives the result.
\end{proof}

\section{Numerical illustrations}\label{sec: Numerical Illustrations}

We illustrate our framework on our two running examples, the \TA and \ST. We first describe the common experimental protocol and the way the WDRO and ERM baselines are compared. We then study the two applications separately, highlighting how the distributionally robust approach adapts to different structural environments: the continuous \TA in \Cref{subsec: uncertain ta} and the relaxed \ST in \Cref{subsec: uncertain st}.

\subsection{Experimental setting}\label{subsec: experimental setting}

We compare the robustness of solutions obtained with our WDRO framework against solutions obtained by standard empirical risk minimization. Given a training data set $\widehat{\xi}_1,\dots,\widehat{\xi}_{N_{\text{train}}}$, both approaches use the same feasible region and the same linear minimization oracle; they differ only in the objective and in the gradient information used by the first-order method.

\begin{table}[H]
\centering
{\tabulinesep=1.2mm
\begin{tabu} to \linewidth {|c|X[c]|X[c]|}
\hline
 & ERM & WDRO \\ \hline
Objective & $\dfrac{1}{N_{\text{train}}}\displaystyle\sum_{k=1}^{N_{\text{train}}}f(\bfx,\widehat{\xi}_k)$ & $F(\bfx,\lambda)$ as in \Cref{eq: regularized WDRO with duality} \\ \hline
Gradient & $\dfrac{1}{N_{\text{train}}}\displaystyle\sum_{k=1}^{N_{\text{train}}}\nabla_{\bfx}f(\bfx,\widehat{\xi}_k)$ & $\mathcal{G}^{\mathscr{J}}(\bfx,\lambda)$ as in \Cref{eq: gradient estimate} \\ \hline
LMO & \multicolumn{2}{c|}{Same linear oracle} \\ \hline
Algorithm & \emph{Classical Frank-Wolfe} & \Cref{alg: Momentum stochastic FW} \\ \hline
\end{tabu}}
\end{table}

Once both solutions are computed, we evaluate their out-of-sample behavior on shifted test scenarios $\widetilde{\xi}_1,\dots,\widetilde{\xi}_{N_{\text{test}}}$. More precisely, we compare the parametrized losses $f(\bfx_{\text{WDRO}},\cdot)$ and $f(\bfx_{\text{ERM}},\cdot)$ on data drawn from a distribution that differs from the training distribution. The instance generators are chosen to make this shift controlled: for the \ST, we build synthetic graphs and cost matrices; for the \TA, we perturb a real base network. This lets us stress-test ERM while keeping the experimental setting interpretable.

Throughout the experiments, we use the squared Euclidean ground cost
$c(\xi,\zeta)\eqdef\Vert\xi-\zeta\Vert^2_2$
in the Wasserstein distance. The robust model also depends on the Wasserstein radius $\varrho>0$, the sampling variance $\sigma^2$, and the smoothing temperature $\varepsilon>0$. These parameters shape the robust objective itself, so we tune them empirically through validation experiments.

The remaining numerical issue is the upper bound $\lambda_{\max}$ for the dual variable. The bound must be large enough not to exclude the relevant minimizer, but excessively large values can create numerical instabilities in the exponential weights when $\lambda\Vert\widehat{\xi}_k-\zeta\Vert^2_2$ dominates $f(\bfx,\zeta)$. We therefore calibrate $\lambda_{\max}$ so that $\lambda_{\max}\overline{c}$ has the same order of magnitude as the typical dispersion $\Delta_f$ of the loss over sampled scenarios, where
$$
\Delta_f\eqdef\E_{\bfx}\left[\diam(f(\bfx,\Xi))\right]=\E_{\bfx}\left[\max_{\xi\in\Xi}f(\bfx,\xi)-\min_{\zeta\in\Xi}f(\bfx,\zeta)\right]\qquad\text{and}\qquad \overline{c}\eqdef\E_{\widehat{\xi}\sim\widehat{\Proba}_N}\left[\E_{\zeta}\left[\Vert\widehat{\xi}-\zeta\Vert^2_2\right]\right].
$$
Both quantities are only needed as calibration scales, so we estimate them by the following sampling heuristic.

\begin{algorithm}[H]
\caption{\emph{Calibration of $\lambda_{\max}$}}\label{alg: calibration of lambdamax}
\begin{algorithmic}
\Require training samples $\left\lbrace\widehat{\xi}_1,\dots,\widehat{\xi}_N\right\rbrace$ of the learning problem
\For{$k\in\llbracket N\rrbracket$}
\State Sample $\zeta^{(k)}_1,\dots,\zeta^{(k)}_S\sim\normal{\widehat{\xi}_k}{\sigma^2\mathbf{I}}$
\State $\bfx_k\gets \text{LMO}\left(\zeta_k\right)$ with $\zeta_k\sim\normal{\widehat{\xi}_k}{\sigma^2\mathbf{I}}$
\EndFor
\State $\widetilde{c}\gets\dfrac{1}{NS}\displaystyle\sum_{k=1}^N\sum_{s=1}^S \Vert\widehat{\xi}_k-\zeta^{(k)}_s\Vert^2_2$\Comment{Approximation of the average cost}
\State $\widetilde{\Delta_f}\gets\dfrac{1}{N}\displaystyle\sum^N_{k=1}\left(\max_{s\in\llbracket S\rrbracket}\left\lbrace f(\bfx_k,\zeta^{(k)}_s)\right\rbrace-\min_{s'\in\llbracket S\rrbracket}\left\lbrace f(\bfx_k,\zeta^{(k)}_s)\right\rbrace\right)$\Comment{Approximation of the average dispersion}
\State\Return $\lambda_{\max}\eqdef\Quotient{\widetilde{\Delta_f}}{2\widetilde{c}}$
\end{algorithmic}
\end{algorithm}

For all experiments, the primal variable $\bfx_0$ is initialized at random in $\Xset$, and the dual variable is initialized at the midpoint of the calibrated interval, $\lambda_0\eqdef\lambda_{\max}/2$. Both WDRO and ERM solvers are run with a maximum budget of $5000$ iterations.

\subsection{Uncertain \TA}\label{subsec: uncertain ta}

We construct our instances based on the ``SiouxFalls'' network, with 24 nodes and 76 links, from the \emph{TransportationNetwork} dataset \citep{xu2024unified}. For this problem, our linear minimization oracle is solving a shortest path problem on the network, with weights given by the current gradient \citep{mitradjieva2013stiff}. In our experiments, these shortest path subproblems are solved using \emph{Dijkstra's} algorithm, which runs in $\complexity{|A|+n\log(n)}$, where $n$ is the number of nodes of the network. We compute solutions to both models
\begin{equation*}
\bfx_{\text{WDRO}}\eqdef\underset{\bfx\in\Xset_{\mathrm{flow}}}{\argmin}\min_{\lambda\in [0,\lambda_{\max}]}F(\bfx,\lambda)
\qquad\text{and}\qquad \bfx_{\text{ERM}}\eqdef\underset{\bfx\in\Xset_{\mathrm{flow}}}{\argmin}\dfrac{1}{N_{\text{train}}}\sum^{N_{\text{train}}}_{k=1}\displaystyle\sum_{a\in A}\int^{\bfx_a}_0\left(t_k\right)^{(0)}_a\left(1+\alpha_k\left(\dfrac{s}{c_a}\right)^{\beta_k}\right)\mathrm{d}s.\label{eq: erm solution ta}
\end{equation*}
We build a shifted distribution $\widetilde{\mathbf{P}}_{\mathbf{Inst}}$, whose construction is detailed in \Cref{subseq: instances for the ta}, from which we sample a shifted set
$
\Xi_{\text{test}}\eqdef\{\widetilde{\xi}_1,\dots,\widetilde{\xi}_{N_{\text{test}}}\}
$
with $\widetilde{\xi}_1,\dots,\widetilde{\xi}_{N_{\text{test}}}\sim\widetilde{\mathbf{P}}_{\mathbf{Inst}}$. We evaluate the losses on the training dataset and the test dataset
$$
\mathscr{L}_{\Xi}(\bfx)\eqdef
\left\lbrace
f(\bfx,\xi)=\bfx^\top\xi \bfx
\,\middle|\,
\xi\in\Xi\right\rbrace,
\qquad
\bfx\in\{\bfx_{\text{ERM}},\bfx_{\text{WDRO}}\},
\quad
\Xi\in\{\Xi_{\text{train}},\Xi_{\text{test}}\}.
$$
For this experiment, the chosen WDRO parameters are: $\varepsilon=10^{-3}$ for the smoothing temperature, 
the Wasserstein radius has been set to $\varrho=17$, and we estimated each integral with $S=100$ samples and a mini-batch size of $b=10$.

\begin{figure}[H]
\centering
\begin{subcaptiongroup}
	\begin{minipage}{0.45\linewidth}
       \centering
       \includegraphics[width=1\linewidth,keepaspectratio]{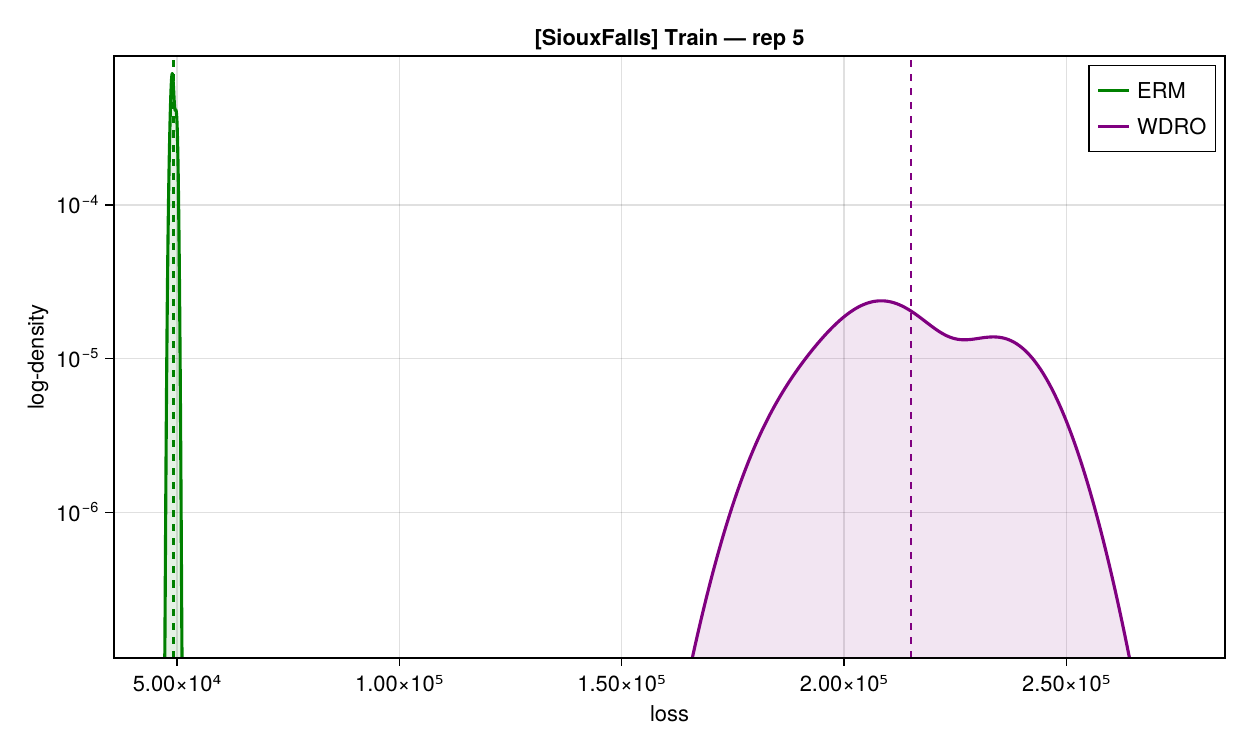}
       \caption{training data}
       \label{Fig: TA training test example}
   	\end{minipage}\hfill
    \begin{minipage}{0.45\linewidth}
       \centering
       \includegraphics[width=1\linewidth,keepaspectratio]{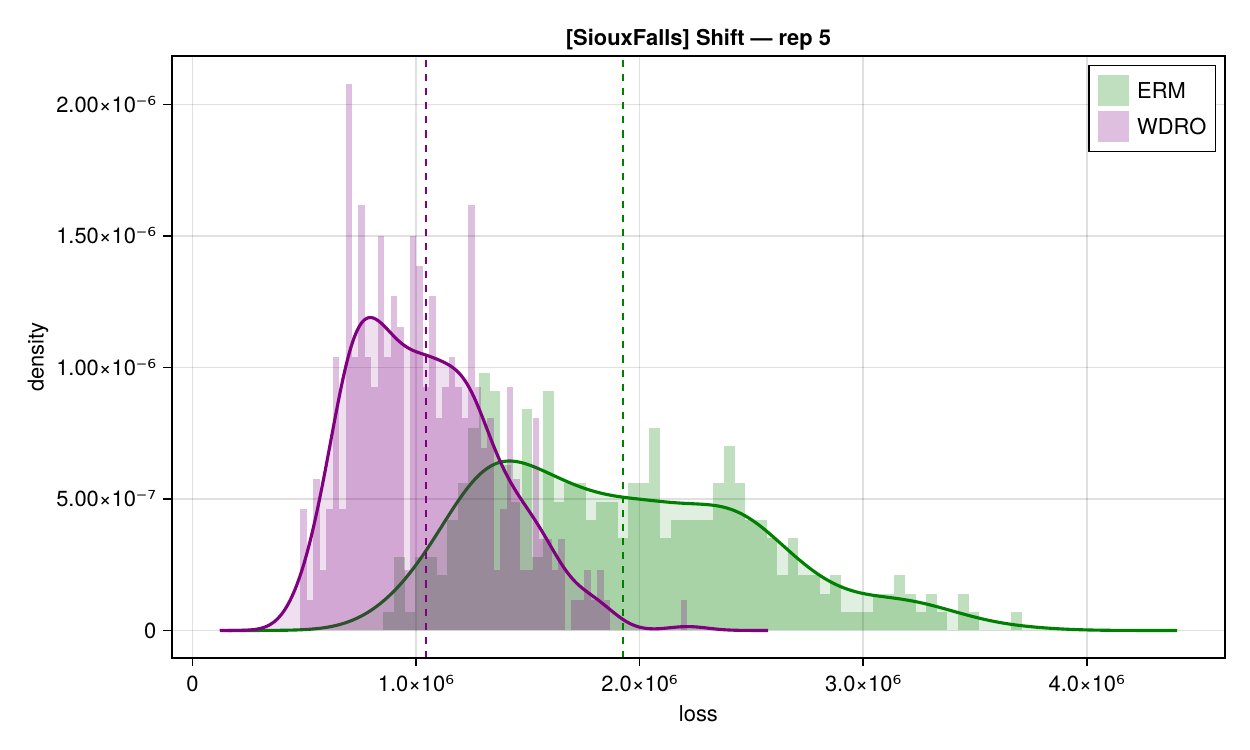}
       \caption{shift test}
       \label{Fig: TA shift test example}
   \end{minipage}
   \end{subcaptiongroup}
   \caption{example of \TA on ``SiouxFalls''. Purple histograms represent the losses obtained with the robust solution; green histograms represent the losses obtained with the ERM. For better readability, in \Cref{Fig: TA training test example} we represented the log-density of the Kernel density of the losses on the training data.}
\end{figure}

\Cref{Fig: TA training test example,Fig: TA shift test example} illustrate that on the training data, the ERM approach outperforms the robust method, achieving an overly optimistic low travel time. Conversely, the WDRO objective function leads to much higher total travel times on the training set, of the order of $2.18\times 10^5$. On the shifted test data, the ERM severely degrades and underperforms, with an average loss of $1.92\times 10^6$, while the robust solution remains stable and achieves a lower out-of-sample travel time, with an average objective of $1.04\times 10^6$.

\begin{figure}[H]
\centering
\begin{subcaptiongroup}
	\hfill
	\begin{minipage}{0.3\linewidth}
       \centering
       \includegraphics[width=1\linewidth,keepaspectratio]{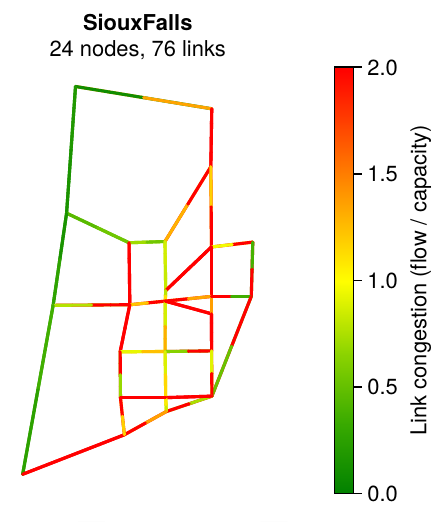}
       \caption{shift ERM flow}
       \label{Fig: TA erm test flow}
   	\end{minipage}\hfill
   \begin{minipage}{0.3\linewidth}
       \centering
       \includegraphics[width=1\linewidth,keepaspectratio]{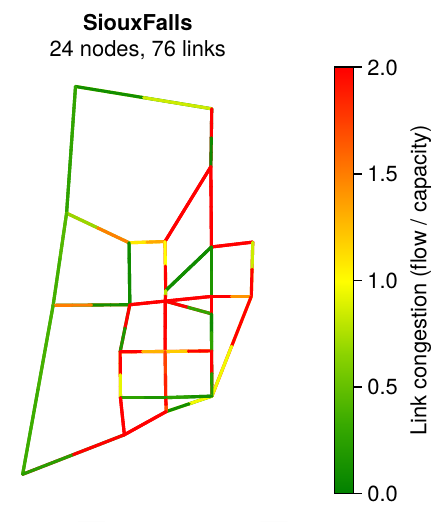}
       \caption{shift robust flow}
       \label{Fig: TA wdro test flow}
   \end{minipage}
   \hfill
   \end{subcaptiongroup}
   \caption{Illustrate the behavior of the ERM solution (left) and the robust solution (right) on ``SiouxFalls''. On each arc, a red shade indicates a high congestion-to-capacity ratio, while a green shade indicates that the considered link is uncongested.}
\end{figure}

In \Cref{Fig: TA erm test flow,Fig: TA wdro test flow},
we visualize the routing profiles yielded by the ERM and the robust solution on a common instance with a shifted distribution. We see that the robust solution alleviates congestion on many arcs in the city center, providing a solution with a reduced objective value.

\subsection{Uncertain \ST}\label{subsec: uncertain st}

For the \ST, we evaluate the robustness and out-of-sample performance of the fractional solutions produced by the convex relaxation. The instances are generated with varying degrees of freedom, through changes in the number of nonzero entries, the variance of the entries, and the support pattern; the full generation procedure is described in \Cref{eq: generation of instances for the ST} and \Cref{alg: generation of ST instances}. The linear minimization oracle is solved by Kruskal's algorithm \citep{schrijver2002combinatorial}, with complexity $\complexity{m\log(m)}$.

To test robustness, we build a shifted distribution $\widetilde{\mathbf{P}}_G$ using the same construction as the training distribution, but with a modified ground cost $\widetilde{\mu}_G$, a modified mask $\widetilde{M}$, and a larger variance on the entries. This shift is designed to represent a more unstable regime, where the test instances have more degrees of freedom than the training instances. We then compute the relaxed WDRO and ERM solutions
$$
\bfx_{\text{WDRO}}\in\underset{\bfx\in\Xset_{\mathrm{tree}}}{\argmin}\min_{\lambda\in [0,\lambda_{\max}]}F(\bfx,\lambda),
\qquad
\bfx_{\text{ERM}}\in\underset{\bfx\in\Xset_{\mathrm{tree}}}{\argmin}
\dfrac{1}{N_{\text{train}}}\sum^{N_{\text{train}}}_{k=1}\bfx^\top\widehat{\xi}_k\bfx.
$$

In practice, we set $\varepsilon=10^{-4}$, 
$\varrho=10$, and use a sampling budget $S=10$ with mini-batches of size $b=10$. The shifted test set is sampled as
$
\Xi_{\text{test}}\eqdef\{\widetilde{\xi}_1,\dots,\widetilde{\xi}_{N_{\text{test}}}\}
$
with $\widetilde{\xi}_1,\dots,\widetilde{\xi}_{N_{\text{test}}}\sim\widetilde{\mathbf{P}}_G$. We compare ERM and WDRO on both the training and test sets through the loss collections 
$$
\mathscr{L}_{\Xi}(\bfx)\eqdef
\left\lbrace
f(\bfx,\xi)=\bfx^\top\xi \bfx
\,\middle|\,
\xi\in\Xi\right\rbrace,
\qquad
\bfx\in\{\bfx_{\text{ERM}},\bfx_{\text{WDRO}}\},
\quad
\Xi\in\{\Xi_{\text{train}},\Xi_{\text{test}}\}.
$$
The histograms in \Cref{Fig: revised training test example,Fig: revised shift test example} show the resulting empirical loss distributions, with kernel-density estimates added for readability.

\begin{figure}[H]
\centering
\begin{subcaptiongroup}
	\begin{minipage}{0.5\linewidth}
       \centering
       \includegraphics[width=1\linewidth,keepaspectratio]{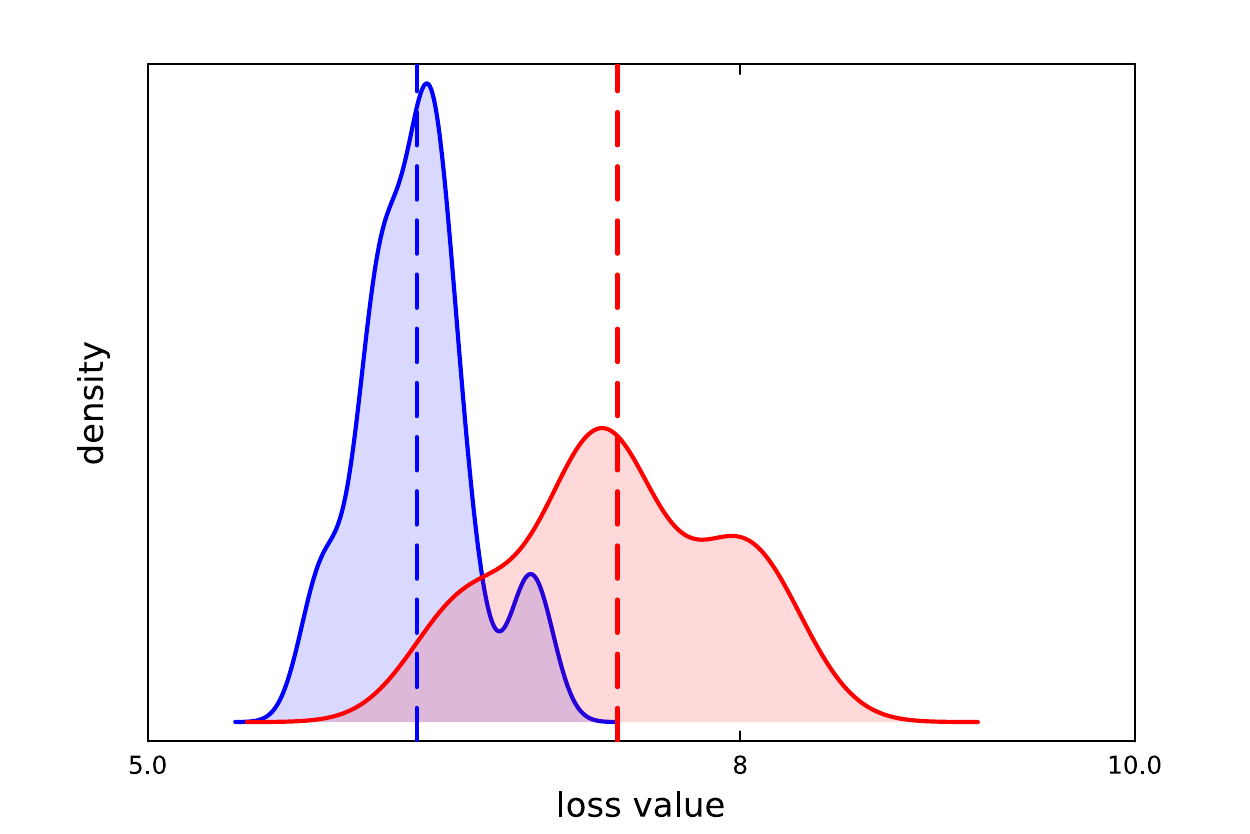}
       \caption{training data}
       \label{Fig: revised training test example}
   	\end{minipage}\hfill
       \begin{minipage}{0.5\linewidth}
       \centering
       \includegraphics[width=1\linewidth,keepaspectratio]{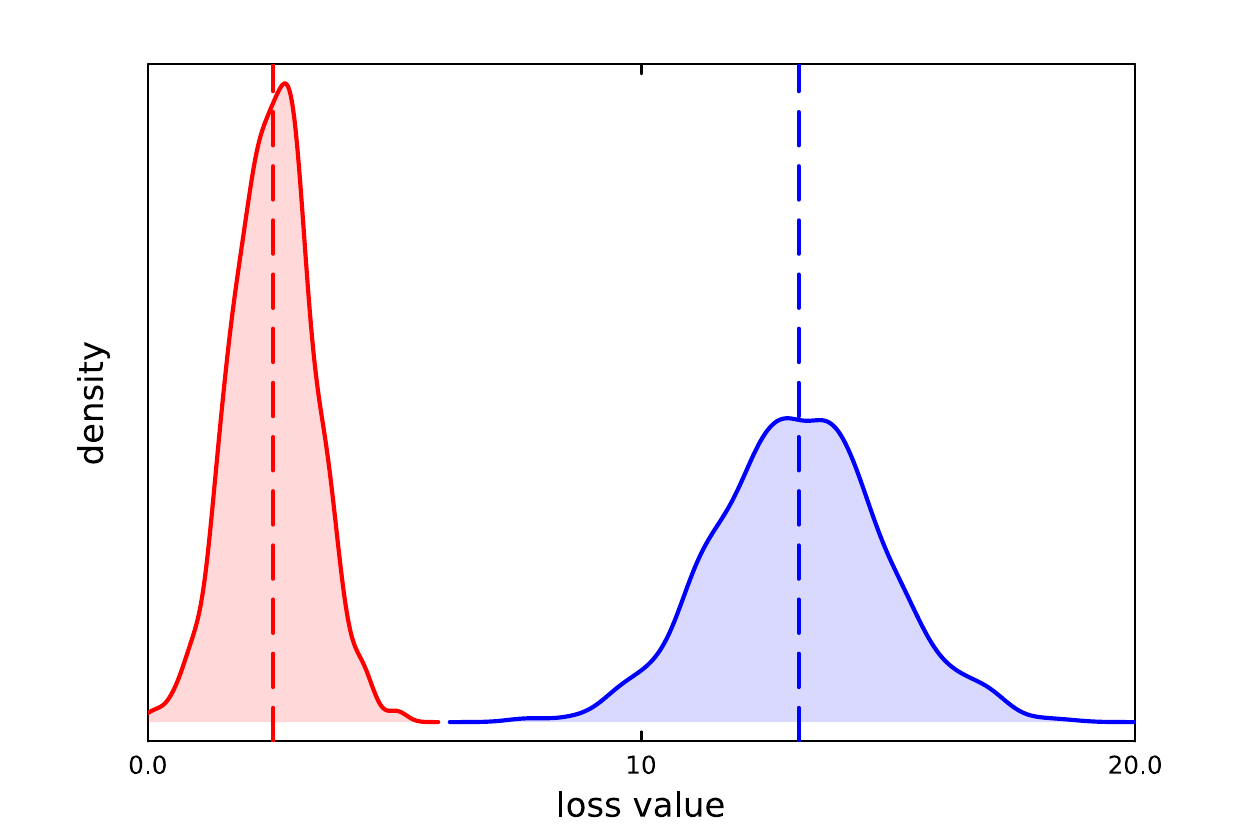}
       \caption{shift test}
       \label{Fig: revised shift test example}
   \end{minipage}
   \end{subcaptiongroup}
   \caption{\ST example on a graph $G$ on $n=50$ vertices and $m=331$ edges. Red smoothed histograms represent the losses obtained with the robust solution; blue histograms represent the losses obtained with the ERM. The ERM solution is overconfident and underperforms on novel data.}
\end{figure}

On the training data, \Cref{Fig: revised training test example} shows the expected advantage of ERM: it is optimized directly on the empirical distribution and therefore obtains the lowest average training losses, while the WDRO solution is deliberately more conservative. Under the shifted distribution, however, the ranking reverses\footnote{This behavior should be interpreted in light of \Cref{sec: Properties of ERM}. In highly constrained combinatorial problems, the feasible set may contain only a limited number of solutions, so ERM can already generalize well even under moderate distributional shifts. In particular, \Cref{thm:ermgeneralize,thm:erm_shift} establish the good performance of ERM in finite settings. The robust objective is more useful in looser settings, where computing an integer solution is in practice intractable, or where the feasible set is richer, and ERM has more opportunity to overfit the empirical sample. Consistently with \Cref{thm:ermgeneralize}, we study how the relative performance of WDRO changes as the problem becomes less constrained, using the graph size $n$ and the density $d(m,n)\eqdef 2m/(n(n-1))$ as complexity parameters.}. As shown in \Cref{Fig: revised shift test example}, the WDRO solution achieves smaller losses than ERM, illustrating the intended tradeoff: WDRO sacrifices empirical performance to obtain a solution that is less sensitive to distributional changes.

For a second experiment, we focus on a density test with fixed graph size $n=50$. For densities spread from sparse to complete graphs, we consider edge counts $m_k$ such that $d(m_k,n)\approx k/10$, $k\in\llbracket 10\rrbracket$. For each density, we generate $n_{\text{inst}}=10$ independent instances on the same graph size but with different training dynamics, namely different base costs, masks, and generation parameters. We aggregate the shifted-test advantage of WDRO over ERM through
$$
\mathscr{D}^{(m)}_{\text{text}}\eqdef
\bigcup_{i=1}^{n_{\text{inst}}}
\left\lbrace
f({\bfx_{\text{ERM}}^{(m,i)}},\xi)
-
f({\bfx_{\text{WDRO}}^{(m,i)}},\xi)
\,\middle|\,
\xi\in\Xi_{\text{test}}
\right\rbrace,
$$
where ${\bfx_{\text{ERM}}^{(m,i)}}$ and ${\bfx_{\text{WDRO}}^{(m,i)}}$ denote the ERM and WDRO solutions for the $i^{\text{th}}$ instance at edge count $m$.
\begin{wrapfigure}[18]{r}{0.5\textwidth}
  \centering
  \includegraphics[width=\linewidth,keepaspectratio]{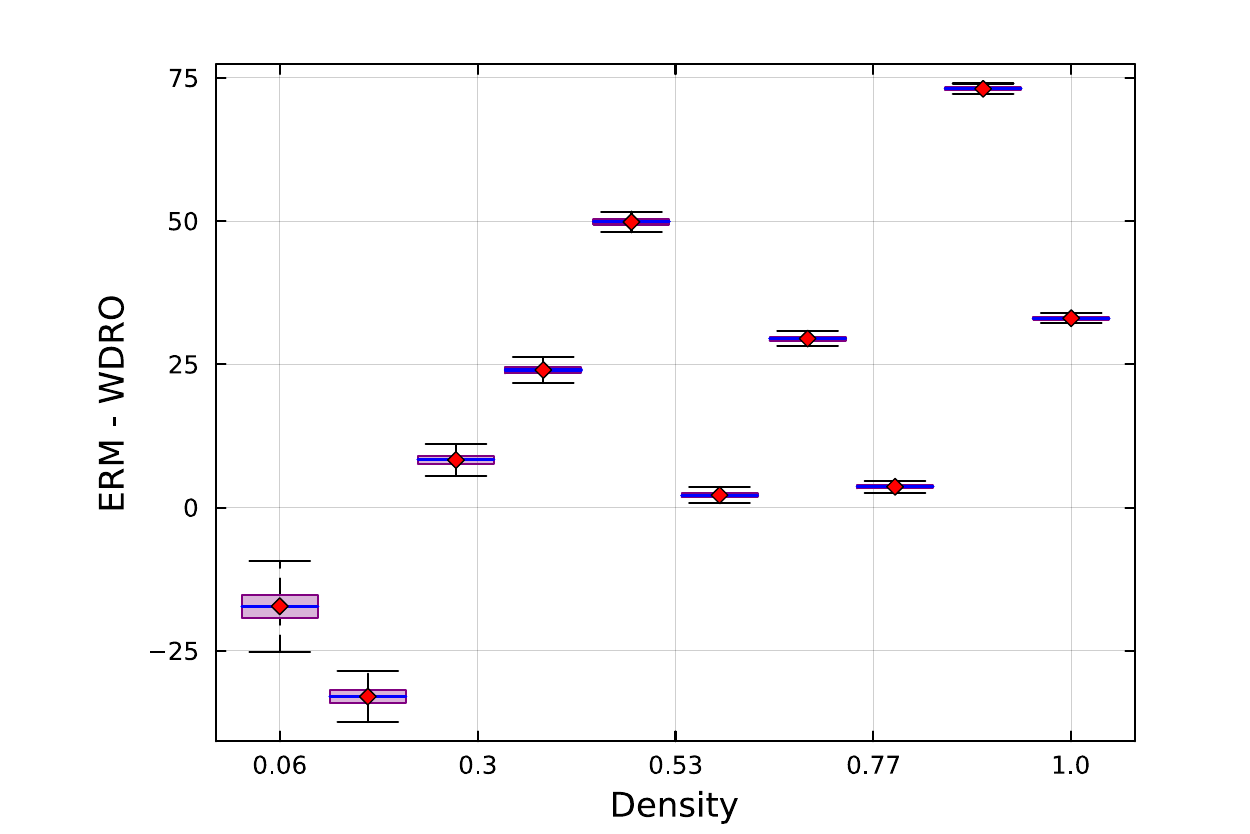}
  \caption{aggregated losses for instances of the \ST on a graph $G$ on $n=50$ vertices and $m\in\lbrace 75,203,\dots,1225\rbrace$ edges.}\label{Fig: revised shift test summary}
\end{wrapfigure}
\Cref{Fig: revised shift test summary} reports these aggregated shifted losses. The distributions remain instance-dependent, and their averages do not follow a perfectly monotone trend, reflecting the sensitivity of robustness gains to the particular shifted distribution $\widetilde{\mathbf{P}}_G$. Nevertheless, the averages of $\mathscr{D}^{(m)}_{\text{test}}$ are positive across densities, meaning that WDRO improves shifted-test performance on average. The dispersion also narrows as density increases: in denser, less constrained graphs, the advantage of the robust solution becomes more systematic. Overall, the experiment suggests that the benefit of WDRO appears when the shifted distribution is sufficiently different from the empirical distribution, especially when the training set has limited variability or when the scenario variance is large enough to expose different behaviors in the gradient estimates. This is precisely the regime where distributional shifts can change the effective decision problem, and where the WDRO formulation is expected to be most relevant.

\section{Conclusion}

This work proposes a general algorithmic framework bringing Wasserstein distributional robustness to constrained stochastic optimization. The main ingredient is to replace the nonsmooth inner worst-case term by an entropic approximation, which yields a smooth objective with computable stochastic gradients.
Coupled with a momentum stochastic Frank-Wolfe scheme, this makes the approach compatible with feasible sets described through a linear minimization oracle, and therefore with many continuous constraint sets, as well as convex relaxations of mixed-integer and combinatorial problems.
We prove the consistency and unbiasedness of the gradient estimators and derive convergence guarantees for the resulting method.
The numerical experiments on traffic assignment and quadratic spanning trees show the expected tradeoff: ERM remains competitive when the constraints restrict the solutions, whereas WDRO becomes valuable when the test distribution departs from the training data and the feasible set is rich enough for overfitting.
This supports the use of smooth WDRO as a practical robustness tool for structured decision problems, and opens many questions, notably on geometric properties of the feasible set which determine the empirical gains of robustness.

\section{Acknowledgments}


This work was supported by the LabEx \textsc{Persyval}-Lab (ANR-11-LABX-0025) and by MIAI @ Grenoble Alpes (ANR-19-P3IA-0003), whose support is gratefully acknowledged.
We thank Changhyun Kwon and Guillaume Dalle for the development of \texttt{TrafficAssignment.jl}. We also thank Mathis Azéma for relevant discussions on 
the unconstrained setting, and extend our sincere thanks to Waïss Azizian for his feedback and comments on the paper.

\bibliography{refs}

\appendix
\newpage
\crefalias{section}{appendix}
\crefalias{subsection}{appendix}

\section{Generation of instances}

This appendix details the instance-generation procedures used in the numerical section. Its purpose is twofold: first, to make explicit which part of each problem instance is kept fixed and which part is random; second, to clarify how the training and shifted distributions differ in the experiments. The two procedures below follow the same pattern: a nominal structure is fixed, and random scenarios are generated around it. In \Cref{sec: Numerical Illustrations},  the robust and ERM solutions are then compared on samples drawn either from the training distribution or from a shifted distribution.

\subsection{Instances for the Traffic Assignment}\label{subseq: instances for the ta}

The instances used for the \TA experiments come from the \emph{TransportationNetwork} dataset (\href{https://github.com/bstabler/TransportationNetworks}{\textcolor{citationcolor}{\texttt{https://github.com/bstabler/TransportationNetworks}}}). Each nominal instance provides a network $G=(V,A)$ with capacities $c_a>0$, free-flow times $t^{(0)}_a>0$ for all $a\in A$, along with a congestion multiplier $\alpha$ and a power parameter $\beta>0$ in the BPR-type travel-time function.

The network topology and capacities are kept fixed throughout the experiment. Uncertainty is introduced only through the travel-time parameters: free-flow times are multiplied arc-wise, while $\alpha$ and $\beta$ are sampled around shifted values. Positive values of $\mu_{\alpha}$ and $\mu_{\beta}$ make the training scenarios optimistic relative to the nominal parameters, which creates a controlled mismatch with the more pessimistic shifted scenarios used for out-of-sample evaluation. The uncertain training distribution is generated as follows:
\begin{algorithm}[H]
\caption{\emph{Generation of Uncertain \TA training instances}}\label{alg: generation of TA train instances}
\begin{algorithmic}
\Require $\mathbf{Inst}=\left(G=(V,A),\left(c_a\right)_{a\in A},(t^{(0)}_a)_{a\in A},\alpha,\beta\right)$ the nominal instance
\State Uncertainty parameters:
\State\hspace*{1\parindent}$\sigma_{\alpha}^2>0$\Comment{controls the variance of the uncertainty on the multiplier}
\State\hspace*{1\parindent}$B_{\beta}$>0\Comment{length of the sampling interval for the power parameter}
\State\hspace*{1\parindent}$\mu_{\alpha},\mu_{\beta}$: shifting the training parameters\Comment{$\mu_{\alpha},\mu_{\beta}>0$ shifts towards an optimistic setting}
\State Build distribution $\mathbf{P}_{\mathbf{Inst}}$ defined as
\State\hspace*{1\parindent}\textbf{function }$\mathbf{P}_{\mathbf{Inst}}$:
\State\hspace*{2\parindent}\textbf{for} $a\in A$ \textbf{do}
\State\hspace*{3\parindent}Sample $m_a\sim\mathcal{U}\left([0.75,1]\right)$\Comment{change the free flow time of arc $a$}
\State\hspace*{2\parindent}\textbf{end for}
\State\hspace*{2\parindent} Sample $\widetilde{\alpha}>0$ with $\widetilde{\alpha}\sim\normal{\alpha-\mu_{\alpha}}{\sigma_{\alpha}^2}$\Comment{normal perturbation}
\State\hspace*{2\parindent} Sample $\widetilde{\beta}>0$ with $\widetilde{\beta}\sim\mathcal{U}\left(\left[\left(\beta-\mu_{\beta}\right)\pm\Quotient{B_{\beta}}{2}\right]\right)$\Comment{uniform distribution for the power}
\State\hspace*{2\parindent}\Return  $\mathbf{Inst}_{\text{sample}}=\left(G=(V,A),\left(c_a\right)_{a\in A},(m_{\alpha}\times t^{(0)}_a)_{a\in A},\widetilde{\alpha},\widetilde{\beta}\right)$
\State\hspace*{1\parindent}\textbf{end}
\State\Return $\mathbf{P}_{\mathbf{Inst}}$
\end{algorithmic}
\end{algorithm}

This procedure yields uncertain \TA scenarios in which the objective parameters vary while the underlying road network remains unchanged. We optimize the robust objective and the ERM objective using training scenarios
$\widehat{\xi}_1,\dots,\widehat{\xi}_{N_{\text{train}}}\sim\mathbf{P}_{\mathbf{Inst}}.$
In contrast, the shifted scenarios are generated from a shifted distribution using the same method but with significantly more pessimistic parameters. Capacities are also slightly reduced. Typically, the results illustrated in \Cref{Fig: TA training test example,Fig: TA shift test example} and the illustration of \Cref{Fig: TA erm test flow,Fig: TA wdro test flow} are obtained with generating parameters $\sigma_{\alpha}^2=0.3\times\alpha$, $B_{\beta}=0.9$, $\mu_{\alpha}=0.15\times\alpha$, $\mu_{\beta}=1$.

\subsection{Instances for the Quadratic Minimum Spanning Tree}\label{eq: generation of instances for the ST}

For the \emph{Uncertain \ST}, the graph determines the combinatorial structure, while the random cost matrix determines the quadratic interaction between selected edges.
We first generate a connected graph and then sample a base cost matrix on its fixed edge set.
A binary mask is used to introduce sparsity and missing interactions, so that each scenario contains both noisy costs and incomplete information. The normalization step keeps the scale of the generated matrices comparable across instances. The full generation protocol is given in \Cref{alg: generation of ST instances}.
Each instance is described by a graph $G$ with $n$ vertices and $m$ edges, together with a ``true'' distribution $\mathbf{P}_G$ from which the training data are sampled:
$\widehat{\xi}_1,\dots,\widehat{\xi}_{N_{\text{train}}}\sim\mathbf{P}_G.$

\begin{algorithm}[H]
\caption{\emph{Generation of Uncertain \ST instances}}\label{alg: generation of ST instances}
\begin{algorithmic}
\Require $n$ the number of nodes, $m$ the number of edges
\State $G\gets$ random graph on $n$ vertices and $m$ edges, generated by Erdős–Rényi's algorithm
\While{$G$ is not connected}
\State $G\gets\text{Erdős–Rényi}(n,m)$
\EndWhile
\State Sample $\mu_{G}\in\R^{m\times m}_+$ random\Comment{Base cost}
\State Sample $M\in\lbrace 0,1\rbrace^{m\times m}$, with $M_{ij}\sim\mathcal{B}(0,7),\forall i,j\in\llbracket m\rrbracket$\Comment{30\% chance of missing data}
\State Build distribution $\mathbf{P}_{G}$ defined as
\State\hspace*{1\parindent}\textbf{function }$\mathbf{P}_{G}$:
\State\hspace*{2\parindent} Sample $\mathbf{C}$ with $\mathbf{C}_{ij}\sim\normal{0}{1}$\Comment{perturbation}
\State\hspace*{2\parindent} $N\gets M\otimes\left(\mu_G+0,1\times\mathbf{C}\right)$\Comment{Base cost + noise + missing data}
\State\hspace*{2\parindent}\Return $N/\Vert N\Vert_2$\Comment{normalized to facilitate comparison across instances}
\State\hspace*{1\parindent}\textbf{end}
\State\Return $\left(G,\mathbf{P}_G\right)$
\end{algorithmic}
\end{algorithm}

\section{Comparison between exact and smoothed WDRO}\label{app: exact vs smoothed wdro}

This appendix compares the smoothed WDRO formulation~\Cref{eq: regularized WDRO with duality} and the initial dual WDRO formulation~\Cref{eq: WDRO with duality} in a case where the inner supremum over $\zeta$ can be evaluated explicitly. The aim of this appendix is to show that both the exact and smoothed WDRO approaches may provide improved robustness under distributional shift compared with ERM, but that neither approach uniformly dominates the other.

We focus on the \ST objective with unconstrained cost matrices, namely $\Xi=\R^{m\times m}$, and use the squared Euclidean ground cost $c(\xi,\zeta)=\Vert\xi-\zeta\Vert_2^2$. This setting is simpler than the one considered in \Cref{subsec: uncertain st}, but it provides a useful reference point: the inner supremum in the dual WDRO problem admits a closed-form solution.

For a fixed spanning-tree decision $\bfx$ and a dual parameter $\lambda>0$, the sample-wise adversarial term is
$$
\sup_{\zeta\in\R^{m\times m}}
\left\{
\bfx^\top\zeta\bfx-\lambda\Vert\widehat{\xi}-\zeta\Vert_2^2
\right\}.$$
The function inside the supremum is strictly concave in $\zeta$. Its gradient vanishes at
$
\zeta^\star
=
\widehat{\xi}
+
\frac{1}{2\lambda}\bfx\bfx^\top,
$
which is therefore the unique maximizer. Substituting this optimal value into the dual formulation gives
$
\mathrm{WDRO}_{\mathrm{ST}}
\leq
\min_{\bfx\in\Xset_{\mathrm{tree}}}
\inf_{\lambda>0}
\lambda\varrho
+
\E_{\widehat{\xi}\sim\widehat{\mathbb{P}}_N}
\left[
\bfx^\top\widehat{\xi}\bfx
+
\frac{1}{4\lambda}\Vert\bfx\bfx^\top\Vert^2
\right]
=
\min_{\bfx\in\Xset_{\mathrm{tree}}}
\E_{\widehat{\xi}\sim\widehat{\mathbb{P}}_N}
\left[
\bfx^\top\widehat{\xi}\bfx
+
\inf_{\lambda>0}
\left\{
\lambda\varrho+\frac{\Vert\bfx\Vert^4}{4\lambda}
\right\}
\right].$
The remaining one-dimensional minimization is explicit: the minimizer is
$\lambda^\star=\Vert\bfx\Vert^2/(2\sqrt{\varrho})$
and yields
$$
\mathrm{WDRO}_{\mathrm{ST}}
\leq
\min_{\bfx\in\Xset_{\mathrm{tree}}}
\E_{\widehat{\xi}\sim\widehat{\mathbb{P}}_N}
\left[
\bfx^\top\widehat{\xi}\bfx+\sqrt{\varrho}\Vert\bfx\Vert^2
\right]
=
\min_{\bfx\in\Xset_{\mathrm{tree}}}
\E_{\widehat{\xi}\sim\widehat{\mathbb{P}}_N}
\left[
\bfx^\top\left(\widehat{\xi}+\sqrt{\varrho}\,\mathbf{I}\right)\bfx
\right].$$

Note that this closed-form expression relies on the unconstrained scenario space $\Xi=\R^{m\times m}$ and does not directly apply to the constrained scenario sets considered in the main experiments (see \Cref{eq: generation of instances for the ST}). In contrast, our smoothed approach can accommodate such constraints through sampling feasible scenarios, \eg, using rejection sampling (see \Cref{rem: resample during computation of the gradient}).

Thus, in this unconstrained setting, the exact robust counterpart can be solved as an empirical problem with shifted training matrices $\widehat{\xi}_k+\sqrt{\varrho}\,\mathbf{I}$. We denote the corresponding solution by $\bfx_{\mathrm{exact}}$ and compare it with the smoothed WDRO solution and the ERM solution.

\begin{figure}[H]
\centering
\begin{subcaptiongroup}
	\begin{minipage}{0.33\linewidth}
       \centering
       \includegraphics[width=1\linewidth,keepaspectratio]{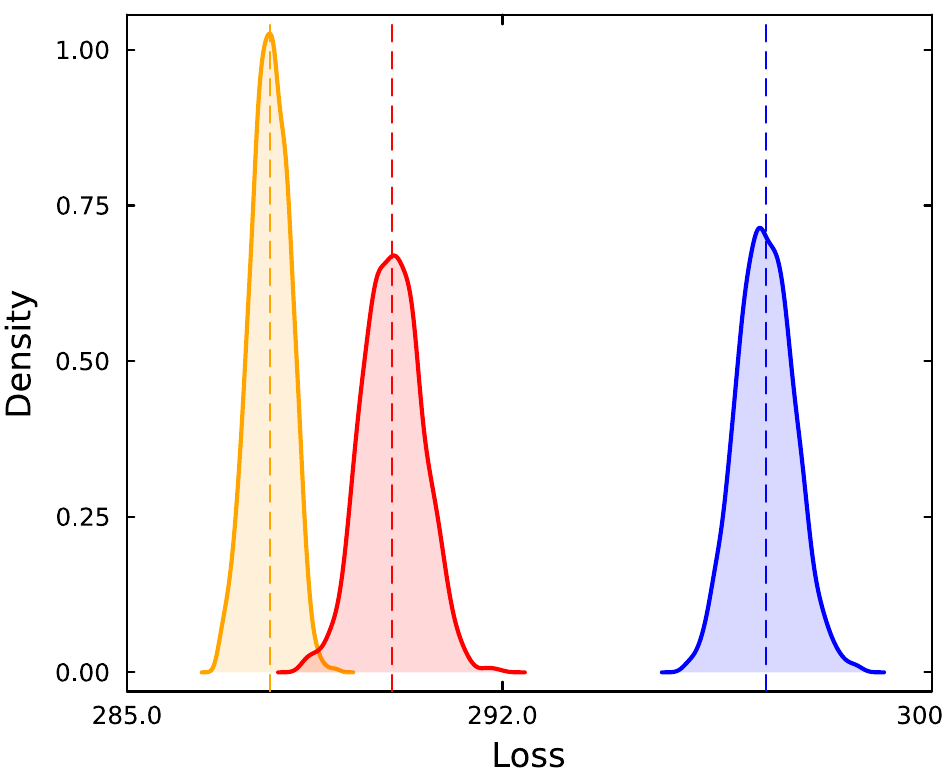}
       \caption{}
       \label{Fig: revised ST Exact vs skwdro1}
   	\end{minipage}\hfill
    \begin{minipage}{0.33\linewidth}
       \centering
       \includegraphics[width=1\linewidth,keepaspectratio]{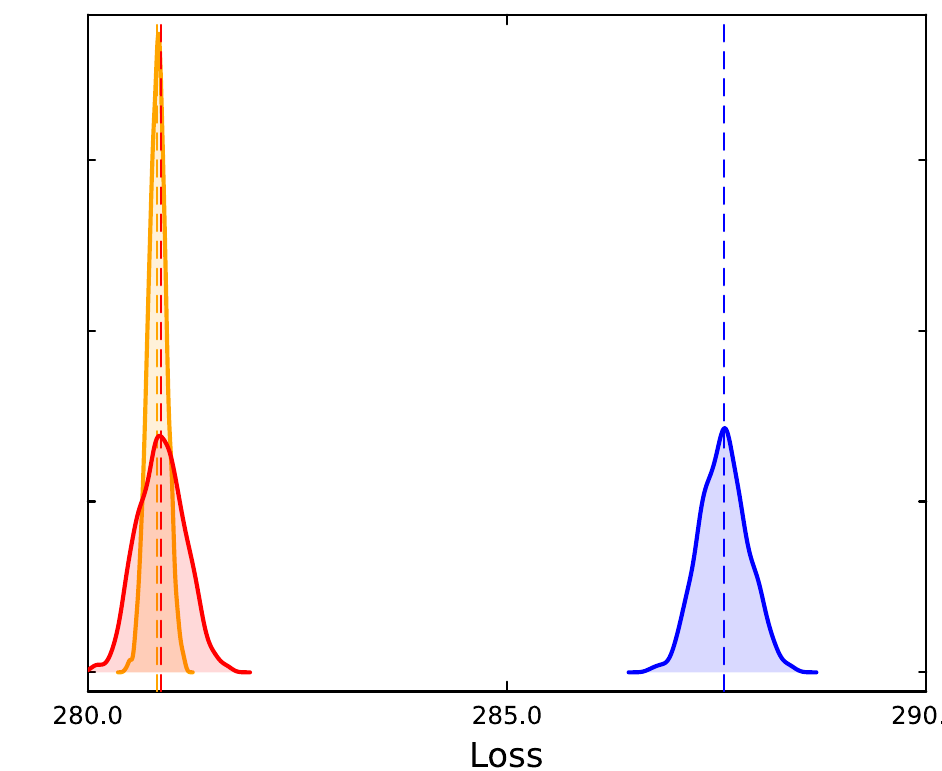}
       \caption{}
       \label{Fig: revised ST Exact vs skwdro2}
   \end{minipage}\hfill
    \begin{minipage}{0.33\linewidth}
       \centering
       \includegraphics[width=1\linewidth,keepaspectratio]{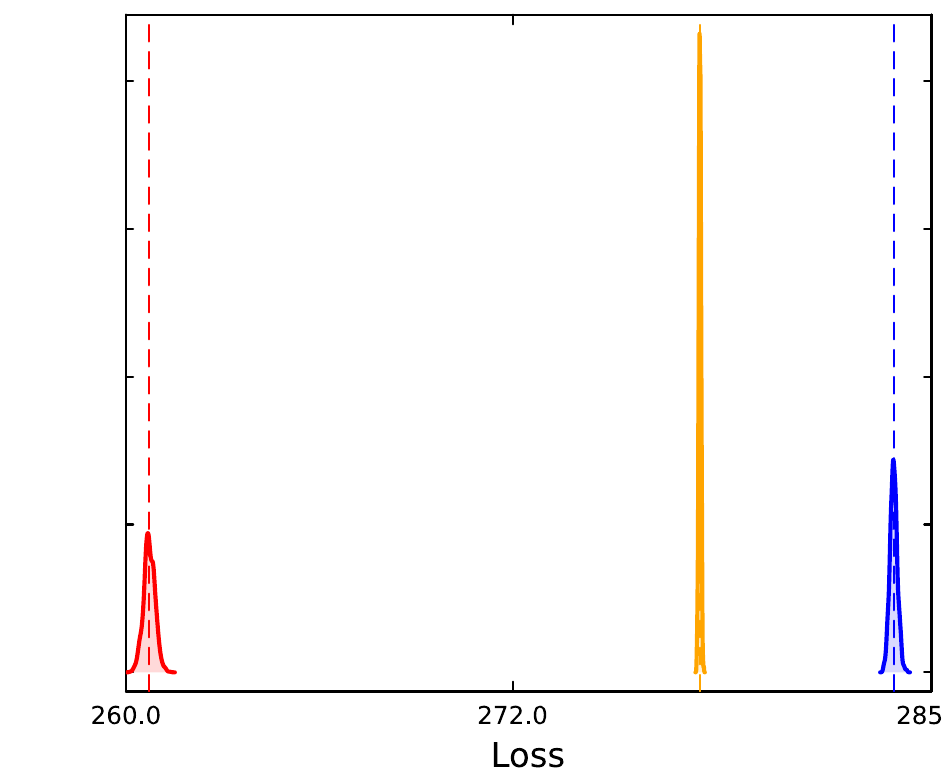}
       \caption{}
       \label{Fig: revised ST Exact vs skwdro3}
   \end{minipage}
   \end{subcaptiongroup}
   \caption{Losses for the \ST with unconstrained cost matrices on unseen data, for 3 different instances (a),(b),(c). Yellow histograms correspond to the exact robust solution, red histograms to the smoothed robust solution, and blue histograms to ERM. We see that both exact/smoothed robust approaches improve over ERM.}
   \label{Fig: revised ST Exact vs skwdro global}
\end{figure}

\hyperref[Fig: revised ST Exact vs skwdro global]{Figure \ref{Fig: revised ST Exact vs skwdro global}} shows that both robust variants improve over ERM and that the comparison between the two is instance-dependent. In some cases, the exact reformulation gives the best shifted performance \hyperref[Fig: revised ST Exact vs skwdro1]{(a)}; in others, the smoothed WDRO solution is comparable \hyperref[Fig: revised ST Exact vs skwdro2]{(b)} or better \hyperref[Fig: revised ST Exact vs skwdro3]{(c)}. This is consistent with the role of smoothing in the main framework: it is not intended to reproduce the exact worst-case solution in every special case, but to provide a tractable and broadly applicable surrogate when the exact inner supremum is unavailable or too costly to exploit.

\section{Properties of the empirical risk minimization in the discrete case}\label{sec: Properties of ERM}

In this section, we quantify the quality of a solution to \Cref{eq: ERM} for the underlying expectation minimization problem when the number of solutions is finite (i.e., $\Xset$ is a finite subset of $\Z^n$.
This highlights a complementary aspect of the behavior of ERM: it tends to perform well in generalization and robustness to moderate shifts in highly constrained settings, similarly to what we observe empirically in the main text for continuous problems.

\Cref{thm:ermgeneralize} formalizes the generalization guarantee for finite solution sets.
The analysis is similar to the ERM guarantees provided in \citet{kleywegt2002sample}, with a few differences.
We use the boundedness of $f$ as a simplifying assumption, which, although stronger than their setting using only finiteness of the moment-generating function of $f(x,\cdot)$, is realistic for many operations research problems.
This allows us to derive a non-asymptotic result that remains valid for any $N$.
We then extend this to the performance of ERM under a moderate distributional shift in \Cref{thm:erm_shift}.

\begin{Theorem}[ERM generalizes in finite sets]\label{thm:ermgeneralize}
Let $\Xset$ be a finite set of points. Let $\mathbf{P}$ be an unknown probability distribution over $\Xi$ and let $\widehat{\xi}_1, \dots, \widehat{\xi}_N$ be drawn i.i.d.~from $\mathbf{P}$.
Define the true risk and empirical risk as
\begin{align*}
\mathbf{F}(\bfx) \eqdef \mathbb{E}_{\xi \sim \mathbf{P}}[f(\bfx, \xi)], \qquad\text{and}\qquad \widehat{\mathbf{F}}_N(\bfx)\eqdef \dfrac{1}{N} \sum_{k=1}^N f(\bfx, \widehat{\xi}_k).
\end{align*}
Define the ERM solution $\bfx_{\mathrm{ERM}}\in\argmin_{\bfx\in\Xset}\widehat{\mathbf{F}}_N(\bfx)$ and the true optimal value $\mathbf{F}^\star\eqdef\min_{\bfx\in\Xset}\mathbf{F}(\bfx)$. Then, for a threshold $\delta \in (0,1)$, with probability at least $1 - \alpha$:
\begin{align}
\mathbf{F}(\bfx_{\mathrm{ERM}}) - \mathbf{F}^\star\leq 2\Vert f\Vert_\infty \sqrt{\frac{2\log\left(\Quotient{2|\Xset|}{\alpha}\right)}{N}} + \delta.\label{eq:generalizationbound}
\end{align}
\end{Theorem}
\begin{proof}
Fix any $\bfx\in \Xset$, the random variables $\left\lbrace f(\bfx,\widehat{\xi}_i)\right\rbrace_{i \in\llbracket N\rrbracket}$ are i.i.d.~with common expectation $\mathbf{F}(\bfx)$ and take values in the interval $[-\Vert f\Vert_\infty,\Vert f\Vert_\infty]$. For any $t > 0$, we can apply Hoeffding's inequality \citep{hoeffding1963probability}:
\begin{align*}
\Proba\left( |\widehat{\mathbf{F}}_N(\bfx) - \mathbf{F}(\bfx)| \geq t \right) \leq 2\exp\left( -\frac{Nt^2}{2\Vert f\Vert_\infty^2} \right).
\end{align*}
By applying the union bound to the event of a deviation $t$ between the empirical and true risks over all $\bfx \in \Xset$,
\begin{align}
\Proba\left( \sup_{\bfx\in \Xset} |\widehat{\mathbf{F}}_N(\bfx) - \mathbf{F}(\bfx)| \geq t \right) & \leq  \sum_{\bfx\in \Xset}\Proba\left( |\widehat{\mathbf{F}}_N(\bfx) - \mathbf{F}(\bfx)| \geq t \right)\label{eq:chain1event} \\
& \leq \sum_{\bfx\in \Xset} 2\exp\left( -\frac{Nt^2}{2\Vert f\Vert_\infty^2} \right) = 2 |\Xset| \exp\left( -\frac{Nt^2}{2\Vert f\Vert_\infty^2} \right) \nonumber.
\end{align}
We can derive the requirement to upper-bound the right-hand side by $\alpha$:
\begin{align*}
2|\Xset| \exp\left( -\frac{Nt^2}{2\Vert f\Vert_\infty^2} \right) \leq \alpha \quad \Rightarrow \quad t \geq \Vert f\Vert_\infty \sqrt{\frac{2\log\left(\Quotient{2|\Xset|}{\alpha}\right)}{N}}.
\end{align*}
Define the uniform convergence event:
\begin{align}
\mathcal{E} = \left\{\sup_{\bfx\in\Xset} |\widehat{\mathbf{F}}_N(\bfx) - \mathbf{F}(\bfx) | \leq \Vert f\Vert_\infty\sqrt{\frac{2\log\left(\Quotient{2|\Xset|}{\alpha}\right)}{N}}\right\}\label{eq:def_unifprobevent}
\end{align}
From complementing the event in the probability in \eqref{eq:chain1event}, we obtain $\Proba(\mathcal{E}) \geq 1-\alpha$.

Denote with $\bfx_{\mathrm{ERM}}$ a $\delta$-approximate optimizer of $\widehat{\mathbf{F}}_N$ over $\Xset$: $\widehat{\mathbf{F}}_N(\bfx_{\mathrm{ERM}}) \leq \delta + \min_{\bfx\in\Xset}\widehat{\mathbf{F}}_N(\bfx)$,
and with $\bfx^\star$ an optimizer of $\mathbf{F}$ over $\Xset$. We have, with probability $1 - \alpha$:
\begin{align}
\mathbf{F}(\bfx_{\mathrm{ERM}}) & \leq \widehat{\textbf{F}}_N(\bfx_{\mathrm{ERM}}) + \Vert f\Vert_\infty \sqrt{\frac{2\log\left(\Quotient{2|\Xset|}{\alpha}\right)}{N}} \label{eq:stepunif1} \\
& \leq \widehat{\mathbf{F}}_N(\bfx^\star) + \delta + \Vert f\Vert_\infty \sqrt{\frac{2\log\left(\Quotient{2|\Xset|}{\alpha}\right)}{N}} \label{eq:stepunif2} \\
& \leq\mathbf{F}(\bfx^\star) + \delta + 2 \Vert f\Vert_\infty\sqrt{\frac{2\log\left(\Quotient{2|\Xset|}{\alpha}\right)}{N}}. \label{eq:stepunif3}
\end{align}
Inequalities \eqref{eq:stepunif1} and \eqref{eq:stepunif3} are obtained from the uniform convergence event valid at
$\bfx_{\mathrm{ERM}}$ and $\bfx^\star$ respectively, and \eqref{eq:stepunif2} comes from $\bfx_{\mathrm{ERM}}$ being a minimizer of $\widehat{\mathbf{F}}_N$.
Rearranging provides \eqref{eq:generalizationbound} with probability $1-\alpha$.
\end{proof}

We now state a technical lemma bounding the change in expectation of a continuous function of a random variable under distributional shift.
The result is akin to the Kantorovich–Rubinstein theorem but without the requirement for the ground cost $c$ to be a metric; this is closely related, \eg, to the weak duality result of \citet{blanchet2019quantifying}.

\begin{Lemma}[Bounding deviations of expectations]\label{lemma:lipschitzineqbound}
Under \Cref{assump: continuity of the cost function for Wasserstein} and \Cref{assump: Xi is compact} (i), suppose $h:\Xi\to\R$ satisfies $|h(\xi) - h(\xi')| \leq \bfL_h \norm{\xi - \xi'}_2$ for all $\xi, \xi' \in \Xi$.
Then for any two distributions $\Q_1,\Q_2$ over $\Xi$:
\begin{align*}
\left|\E_{\xi\sim\Q_1}[h(\xi)] - \E_{\xi'\sim\Q_2}[h(\xi')]\right| \leq \frac{\bfL_h}{\sqrt{\mu_c}} \sqrt{ W_c(\Q_1,\Q_2)}.
\end{align*}
\end{Lemma}
\begin{proof}
Let $\pi$ be a coupling, \ie a joint distribution over $\Xi \times \Xi$ with marginals $\Q_1$ and $\Q_2$, then
\begin{align*}
|\E_{\xi \sim \Q_1}[h(\xi)] - \E_{\xi'\sim\Q_2}[h(\xi')]| = |\E_{(\xi,\xi')\sim\pi}[h(\xi) - h(\xi')]| \leq \E_{(\xi,\xi')\sim\pi}[|h(\xi) - h(\xi')|] \leq \bfL_h \E_{(\xi,\xi')\sim\pi}[ \norm{\xi - \xi'}_2 ].
\end{align*}
Using Jensen's inequality, $\E_{(\xi,\xi')\sim\pi}[ \norm{\xi - \xi'}_2 ] \leq \sqrt{\E_{(\xi,\xi')\sim\pi}[ \norm{\xi - \xi'}_2^2 ]}$.
Combining with the lower bound on the cost from \cref{assump: continuity of the cost function for Wasserstein}, we obtain
\begin{align*}
|\E_{\xi \sim \Q_1}[h(\xi)] - \E_{\xi'\sim\Q_2}[h(\xi')]| \leq \frac{\bfL_h}{\sqrt{\mu_c}} \sqrt{\E_{(\xi,\xi')\sim\pi}[ c(\xi, \xi') ]}.
\end{align*}
The left-hand side does not depend of the chosen coupling $\pi$; taking the infimum of the right-hand side over possible couplings, and using the fact that $\sqrt{\cdot}$ is nondecreasing, we have
\begin{align*}
\inf_{\pi} \sqrt{\E[c]} = \sqrt{\inf_{\pi} \E_{\pi}[c]} = \sqrt{W_c(\Q_1, \Q_2)},
\end{align*}
concluding the proof.
\end{proof}

We can now extend the generalization bound to a performance guarantee of ERM under distributional shift.
For simplicity of exposition, we will consider $\bfx_{\mathrm{ERM}}$, the exact minimizer of the ERM problem, i.e., with $\delta=0$ following the notation of \cref{thm:ermgeneralize}.
\begin{Theorem}[ERM robustness to distribution shift]\label{thm:erm_shift}
Retaining the notation of \Cref{thm:ermgeneralize}, suppose \Cref{assump: Xi is compact}(i) and \Cref{assump: continuity of the cost function for Wasserstein} hold, and let $\bfL_f$ be the Lipschitz constant of $f(\bfx,\cdot)$ valid at all $\bfx \in \Xset$:
\begin{align*}
    |f(\bfx,\xi) - f(\bfx,\xi')| \leq \bfL_f \norm{\xi - \xi'}_2 \quad \forall \xi, \xi' \in \Xi, \bfx \in \Xset.
\end{align*}
Let $\Q$ be a distribution over $\Xi$ such that for some $\varrho > 0$, $W_c(\Proba,\Q) \leq \varrho$.
Define $\mathbf{F}_{\Q}(\bfx)$, $\mathbf{F}(\bfx)$ as the risk of $\bfx$ under distribution $\Q$, $\Proba$ respectively and $\bfx_{\Q} \in \argmin_{\bfx\in\Xset}\mathbf{F}_{\Q}(\bfx)$.
For any $\alpha \in (0, 1)$, with probability at least $1-\alpha$, we have:
\begin{align}
\mathbf{F}_{\Q}(\bfx_{\mathrm{ERM}}) - \mathbf{F}_{\Q}(\bfx_{\Q}) \leq 2\Vert f\Vert_\infty\sqrt{\frac{2\log\left(\Quotient{2|\Xset|}{\alpha}\right)}{N}} + 2\bfL_f\sqrt{ \frac{\varrho}{\mu_c} }.
\end{align}
\end{Theorem}
\begin{proof}
For any fixed $\bfx\in\Xset$, applying \Cref{lemma:lipschitzineqbound} with distributions $\Q,\Proba$ with $W_c(\Proba,\Q) \leq \varrho$
and $h = f(\bfx, \cdot)$,
\begin{align}
|\mathbf{F}(\bfx) - \mathbf{F}_{\Q}(\bfx)| = |\E_{\xi'\sim\Proba}[f(\bfx,\xi')] - \E_{\xi \sim\Q}[f(\bfx,\xi)]| \leq \bfL_f \sqrt{\frac{W_c(\Proba,\Q)}{\mu_c}}  \leq \bfL_f \sqrt{\frac{\varrho}{\mu_c}} \quad \forall \bfx\in \Xset.\label{eq:boundfromlemma}
\end{align}

Using the uniform convergence event $\mathcal{E}$ from \Cref{thm:ermgeneralize}:
\begin{align}
\mathcal{E} = \left\{\sup_{\bfx\in\Xset} |\widehat{\mathbf{F}}_N(\bfx) - \mathbf{F}(\bfx) | \leq \Vert f\Vert_\infty\sqrt{\frac{2\log\left(\Quotient{2|\Xset|}{\alpha}\right)}{N}} \right\}
\end{align}
holding with probability $1-\alpha$, we can derive:
\begin{subequations}
\begin{align}
\mathbf{F}_{\Q}(\bfx_{\mathrm{ERM}}) \leq & \mathbf{F}(\bfx_{\mathrm{ERM}}) + \bfL_f \sqrt{\frac{\varrho}{\mu_c}}  \label{eq:chainwasserstein_1} \\
\leq &\widehat{\mathbf{F}}_N(\bfx_{\mathrm{ERM}}) + \Vert f\Vert_\infty\sqrt{\frac{2\log\left(\Quotient{2|\Xset|}{\alpha}\right)}{N}} + \bfL_f \sqrt{\frac{\varrho}{\mu_c}} \label{eq:chainwasserstein_2} \\
\leq & \widehat{\mathbf{F}}_N(\bfx_{\Q}) + \Vert f\Vert_\infty\sqrt{\frac{2\log\left(\Quotient{2 |\Xset|}{\alpha}\right)}{N}} + \bfL_f \sqrt{\frac{\varrho}{\mu_c}} \label{eq:chainwasserstein_3} \\
\leq & \mathbf{F}(\bfx_{\Q}) + 2\Vert f\Vert_\infty\sqrt{\frac{2\log\left(\Quotient{2 |\Xset|}{\alpha}\right)}{N}} + \bfL_f \sqrt{\frac{\varrho}{\mu_c}} \label{eq:chainwasserstein_4} \\
\leq & \mathbf{F}_{\Q}(\bfx_{\Q}) + 2\Vert f\Vert_\infty\sqrt{\frac{2\log\left(\Quotient{2 |\Xset|}{\alpha}\right)}{N}} + 2 \bfL_f \sqrt{\frac{\varrho}{\mu_c}} \label{eq:chainwasserstein_5}
\end{align}
\end{subequations}
where \eqref{eq:chainwasserstein_1} comes from applying \eqref{eq:boundfromlemma} at $\bfx_{\mathrm{ERM}}$,
\eqref{eq:chainwasserstein_2} uses the uniform convergence event bound holding at $\bfx_{\mathrm{ERM}}$ for $\Proba$,
\eqref{eq:chainwasserstein_3} uses optimality of $\bfx_{\mathrm{ERM}}$ for $\widehat{\mathbf{F}}_{N}$,
\eqref{eq:chainwasserstein_4} applies the uniform convergence event bound, this time at $\bfx_{\Q}$,
and \eqref{eq:chainwasserstein_5} uses \eqref{eq:boundfromlemma} a second time at $\bfx_{\Q}$.
\end{proof}

The key take-away of \Cref{thm:ermgeneralize} and \Cref{thm:erm_shift} is that if the number of vertices of the feasible region
remains polynomial in the dimension $n$ with maximum degree $k$, we obtain a generalization bounded by a quantity
$\complexity{\sqrt{k\log(n)} N^{-1/2}}$.
This analysis heavily relies on the finiteness of $\Xset$ and on the union bound, and it leaves open the question of a similar bound for mixed-integer sets.

\section{Supporting lemmas}\label{app: technical lemmas}

For the sake of completeness, we provide proofs of elementary lemmas used in the convergence analysis.

\begin{Lemma}\label{Lemma: mini-batch unbiasedness}
Let $x_1,\dots,x_N\in\R^d$. Let $\mathscr{J}\subseteq\llbracket N\rrbracket$ be a random subset of size $b$, chosen uniformly and independently among all subsets of size $b$. Then, the empirical average over $\mathscr{J}$ is an unbiased estimator of the average over $\llbracket N\rrbracket$:
$$\E\left[\frac{1}{|\mathscr{J}|}\sum_{k\in\mathscr{J}} x_k\right] = \frac{1}{N} \sum_{k=1}^N x_k.$$
\end{Lemma}
\begin{proof}
For any $\mathscr{J}\subseteq\llbracket N\rrbracket$, since $\mathscr{J}$ has a known size $b$, we can write $\frac{1}{|\mathscr{J}|}\displaystyle\sum_{k\in\mathscr{J}} x_k$ as $\dfrac{1}{b}\displaystyle\sum^N_{k=1} x_k\mathbb{1}_{k\in\mathscr{J}}$, thus giving, by linearity:
$$\E\left[\frac{1}{b}\sum_{k=1}^N x_k\mathbb{1}_{k\in\mathscr{J}}\right]=\frac{1}{b}\sum_{k=1}^N x_k\E\left[\mathbb{1}_{k\in\mathscr{J}}\right]=\frac{1}{b}\sum_{k=1}^N x_k\Proba\left(k\in\mathscr{J}\right)=\frac{1}{b}\sum_{k=1}^N x_k\dfrac{\binom{N-1}{b-1}}{\binom{N}{b}}=\frac{1}{b}\sum_{k=1}^N x_k\dfrac{b}{N}=\frac{1}{N}\sum_{k=1}^N x_k.$$
\end{proof}

\begin{Lemma}\label{lemma: log-sum-exp lowers the function}
Let \Cref{assump: Xi is compact}(i) hold, let $\Q\in\mathrm{Proba}(\Xi)$ and $g:\Xi\rightarrow\R$ a bounded $\Q$-measurable function. Then
$$\E_{\zeta\sim\Q^g}\left[g(\zeta)\right]\geq\log\left(\E_{\zeta\sim\Q}\left[\exp\left(g(\zeta)\right)\right]\right)$$
where $\Q^g$ is the distribution defined by $\mathrm{d}\Q^g(\zeta)\propto\exp\left(g(\zeta)\right)\mathrm{d}\Q(\zeta)$.
\end{Lemma}
\begin{proof}
Consider the function $t\mapsto\log\left(\E_{\zeta\sim\Q}\left[\exp\left(tg(\zeta)\right)\right]\right)$. This function is convex and differentiable with derivative
$$\partial_t\log\left(\E_{\zeta\sim\Q}\left[\exp\left(tg(\zeta)\right)\right]\right)=\dfrac{\E_{\zeta\sim\Q}\left[g(\zeta)\exp\left(tg(\zeta)\right)\right]}{\E_{\zeta\sim\Q}\left[\exp\left(tg(\zeta)\right)\right]}=\E_{\zeta\sim\Q^{tg}}\left[g(\zeta)\right]$$
with $\mathrm{d}\Q^{tg}(\zeta)\propto\exp\left(tg(\zeta)\right)\mathrm{d}\Q(\zeta)$.
The function is thus above its tangent at $1$:
$$\log\left(\E_{\zeta\sim\Q}\left[\exp\left(tg(\zeta)\right)\right]\right)\geq\E_{\zeta\sim\Q^{g}}\left[g(\zeta)\right](t-1)+\log\left(\E_{\zeta\sim\Q}\left[\exp\left(g(\zeta)\right)\right]\right)\qquad\forall t\in\R.$$
Applying the inequality at $t=0$ provides the desired inequality.
\end{proof}
\end{document}